%% file: NSBC.tex
\documentclass[12pt]{article}
\usepackage{amssymb,amsthm,hyperref,amsmath,bm,amsfonts}
\usepackage{arydshln}
\usepackage{verbatim}
\usepackage{graphicx}
\usepackage{float}

\usepackage{subfigure}
\usepackage{multirow}
\usepackage{booktabs}
\usepackage{color}
\usepackage{tikz,pgfplots}
\usetikzlibrary{external}
\usepackage{appendix}

\newtheorem{theorem}{Theorem}[section]
\newtheorem{lemma}[theorem]{Lemma}
\newtheorem{prop}{Proposition}[section]
\newtheorem{definition}{Definition}[section]%
\newtheorem{assumption}{Assumption}[section]
\newtheorem{remark}{Remark}[section]
\numberwithin{equation}{section}

\newcommand\dpath{data}
\newcommand\bbR{\mathbb{R}}
\newcommand\bbS{\mathbb{S}}
\newcommand\bbN{\mathbb{N}}
\newcommand\ak{\left(\sqrt{\varepsilon}\right)}

\newcommand\pd[2]{\dfrac{\partial {#1}}{\partial {#2}}}
\def\+#1{\boldsymbol{#1}}
\newcommand\ang[1]{\left\langle {#1} \right\rangle}

\allowdisplaybreaks
\usepackage{xparse}
\DeclareDocumentCommand\es{ m g g }{%
	{{#1}%
		\IfNoValueF {#3} {_{#3}}
		\IfNoValueF {#2} { ^{(#2)}}%
    }%
}
\DeclareDocumentCommand\os{ m g g }{%
	{\overline{#1}%
		\IfNoValueF {#3} {_{#3}}
		\IfNoValueF {#2} { ^{(#2)}}%
    }%
}
\DeclareDocumentCommand\vs{ m g g }{%
	{\widehat{#1}%
		\IfNoValueF {#3} {_{#3}}
		\IfNoValueF {#2} { ^{(#2)}}%
    }%
}
\DeclareDocumentCommand\ks{ m g g }{%
	{\widetilde{#1}%
		\IfNoValueF {#3} {_{#3}}
		\IfNoValueF {#2} { ^{(#2)}}%
    }%
}

\title{Construction of boundary conditions for Navier-Stokes equations from the moment system}
\author{
Ruo Li\thanks{School of Mathematical Sciences, Peking University, Beijing 100871, China. Emails:rli@math.pku.edu.cn} \qquad 
Yichen Yang\thanks{School of Mathematical Sciences, Peking University, Beijing 100871, China. Emails:yichen\_yang@pku.edu.cn} \qquad
Yizhou Zhou\thanks{School of Mathematical Sciences, Peking University, Beijing 100871, China. Emails:zhouyz@math.pku.edu.cn}
}

\date{\today}
\begin{document}
\maketitle{}

\begin{abstract}
	This work concerns with boundary conditions (BCs) of the linearized
	moment system for rarefied gases. As the Knudsen number is
	sufficiently small, we analyze the boundary-layer behaviors of the
	moment system by resorting to a three-scale asymptotic expansion.
	The asymptotic analysis casts the flows into the outer solution, 
	the viscous layer and the Knudsen layer.
	Starting from the BCs of the moment system, we propose a matching
	requirement and construct BCs for the Navier-Stokes equations. The
	obtained BCs contain the effect of 
	second-order terms on the velocity slip and temperature jump. For
	the illustrative case of the Couette flow, we prove the validity of the
	constructed BCs through the error estimates. Meanwhile,
	numerical tests are presented to show the performance of the 
	constructed BCs. 
\end{abstract}

\hspace{-0.5cm}\textbf{Keywords:}
\small{Moment system, Boundary conditions, Navier-Stokes equations, Boundary layers}\\


\input{intro}

\section{Background}
	\label{sec:Back}
	\input{Back}

\input{Couette}

\section{General linear case}
\label{sec:AA}
	\input{L_MoM}

	\input{L_AA}

\section{NS equations and slip BCs}
\label{sec:NSBC}
	\input{L_AABC}

\section{Couette flow: numerical verification}
\label{sec:ANE}
	\input{ANE}

\input{conclusion}
\input{App}

\bibliographystyle{abbrv}
\bibliography{ref}

\end{document}

%% file: intro.tex
\section{Introduction}
For gases in the low-density regime and microscales, one should
consider kinetic theory of gases, described by the Boltzmann 
equation \cite{CC1989}. Since the Boltzmann equation is 
a problem in high (seven) dimensions, the direct simulation will lead
to much more computational cost than the hydrodynamic equations. 
Grad proposed the famous moment method \cite{Grad1949} to reduce the
kinetic equation into low-dimensional moment systems. These systems 
are first-order partial differential equations, which may be regarded
as intermediate models between the Boltzmann equation and hydrodynamic
equations. Recently, with the development of hyperbolic regularization
\cite{CaiFanLi1D,CaiFanLi,CaiFanLiFramework,FanKoellermeier,
KoellermeierTorrilhon}, the moment method has attracted more attentions 
and become a powerful tool in the simulation of gas flow.

The rarefaction effects of the gas are often characterized by the 
Knudsen number $\varepsilon = \lambda/L$, with $\lambda$ the mean 
free path length and $L$ the relevant characteristic length. The
Euler equations and the Navier-Stokes (NS) equations can
\cite{Struchtrup} be formally derived from the moment system 
for small $\varepsilon$. For the initial value problems, the rigorous
proofs of the derivations are given in
\cite{DiFanLiZheng,MaZhiting,ZhaoYongLuo} by resorting to the 
structural stability criterion \cite{Yong}. It was proved that 
the $\varepsilon$-dependent solution of the moment system converges
to the solution of the Euler equations as $\varepsilon$ goes to zero.
Moreover, the error estimates \cite{MaZhiting} indicate that the
error between the solution of the moment system and the solution 
of the NS equations is of $o(\varepsilon)$, i.e., the error goes 
to zero faster than $\varepsilon$ as $\varepsilon\rightarrow 0$. 
In this sense, the NS equations are called the first-order 
approximation of the moment system. 

For the simulation in a bounded domain, proper boundary conditions (BCs) 
should be prescribed, which is rather challenging for moment equations
\cite{Struchtrup}. Following the basic idea in \cite{Grad1949}, 
the work \cite{CaiLiQiao} provides a method to derive BCs for the 
moment closure system from the Maxwell BCs \cite{Maxwell} for the 
Boltzmann equation. One contribution of the present work is to further
investigate the BCs for the moment system by showing the boundary-layer 
behaviors. We perform an asymptotic analysis on the moment system
with the Maxwell-type BCs. Since the moment system can be regarded as a bridge 
between the Boltzmann equation and the macro dynamic equations, we try to
connect its BCs with the BCs for the kinetic equation and hydrodynamics
models (Euler or NS equations) as well. Motivated by this, we propose a 
matching requirement which helps to construct second-order slip BCs 
for the NS equations.

For the initial-boundary value problems (IBVPs) with small parameters,
it is necessary to consider the effect of boundary layers in the 
asymptotic analysis. The moment system is a typical hyperbolic relaxation
system \cite{DiFanLiZheng,ZhaoYongLuo}. We refer to the general theory 
of IBVPs for the relaxation systems \cite{Yong2,ZhouYong2,ZhouYong} and
construct an asymptotic solution $W_\varepsilon$ with three scales of 
spatial variables to the moment system. Specifically, we assume that the
domain is the half-space $\{\+x=(x_1,x_2,x_3)\in\mathbb{R}^3:\ x_2>0\}$ with
the boundary $\{x_2=0\}$. Then $W_\varepsilon$ is constructed by
$$
W_\varepsilon\left(t,\+x^w;x_2\right)
=\os{W}\left(t,\+x^w;x_2\right)
+\vs{W}\left(t,\+x^w;\frac{x_2}{\sqrt{\varepsilon}}\right)
+\ks{W}\left(t,\+x^w;\frac{x_2}{\varepsilon}\right),\qquad
\+x^w=(x_1,x_3).
$$ 
Here $\os{W}$ is called the outer solution while $\vs{W}$ and
$\ks{W}$ are boundary-layer corrections satisfying 
$\vs{W}(t,\+x^w;\infty)=\ks{W}(t,\+x^w;\infty)=0$.
Notice that this construction is consistent with the classical 
knowledge of kinetic theory \cite{CC1989}, which indicates that
the Boltzmann equation with the Maxwell BCs allows two-scales 
boundary layers---the viscous boundary layer (with the length
of $\sqrt{\varepsilon}$) and the Knudsen layer (with the length
of $\varepsilon$). 

Having these, we try to derive BCs for the hydrodynamic equations.
Since the Euler equations are the limiting equations of the moment
system as $\varepsilon\rightarrow 0$, it is natural to require
their BCs to be the so-called reduced BCs (the BCs satisfied by
the relaxation limit) of the moment system. The derivation of 
reduced BCs has been developed in \cite{Yong2,ZhouYong2,ZhouYong} 
for general relaxation systems. In this work, we focus on the 
construction of BCs for the NS equations which are the first-order
approximation to the moment system. According to the classical 
theory \cite{Sone2007}, the NS equations allow boundary layers 
with the length of $\sqrt{\varepsilon}$. Thus we expect the NS
equations, together with their BCs, are satisfactory approximations
for the moment system in the $x_2$ and $x_2/\sqrt{\varepsilon}$ 
spatial scales. These two scales correspond to outer solutions
and viscous boundary layers respectively. In other words, we 
construct the BCs by following:
\begin{itemize}
\item Matching requirement: the solution to the NS equations 
	with constructed BCs approximates 
		$\os{W}(t,\+x^w;x_2)
		+\vs{W}(t,\+x^w;x_2/\sqrt{\varepsilon})$ with an 
		error of $o(\varepsilon)$ for sufficiently small $\varepsilon$.
\end{itemize}
Note that the moment system has more variables than the NS equations.
When we compare the solutions between these two systems, only the
variables of the NS equations (the density $\rho$, the macro 
velocity $\+u$ and the temperature $\theta$) are taken into account.

This paper only deals with the linearized moment system and tries to
construct BCs for the linearized NS equations. In order to identify 
the issues in the construction, we first take the Couette flow as an
illustrative example. In this case, the linearized moment system 
can be decoupled into simplified models. We formulate the BCs for 
the (simplified) moment system and prove their well-posedness by 
checking the strictly dissipative condition \cite{BenzoniSerre}. 
By resorting to formal asymptotic expansions, we show how to determine
each coefficient in the asymptotic solution of the moment system.
Furthermore, the validity of such asymptotic solutions is proved 
through an energy estimate. Having these, we turn to construct BCs
for the (simplified) NS equations. Through an error estimate, we
rigorously prove that the constructed Robin-type BCs indeed satisfy
the aforementioned matching requirement. Moreover, some numerical 
tests validate the performance of the constructed BCs. 

Guided by this example, we proceed to consider the general linear
case. For the linearized moment system, we present a modification
of the Grad BCs such that they are maximal positive, which is 
important for the well-posedness of the symmetric hyperbolic 
system \cite{1960Local,Rauch1985,Paolo1996}. Similar to the 
Couette flow case, we derive the asymptotic equations as well as their
BCs. Since this procedure is parallel to the expansions in 
Sone's generalized slip flow theory \cite{Sone2007}, we compare
the coefficients computed from the moment system with those 
in \cite{2019Hat,Sone2007}. Based on the asymptotic solutions, 
we construct BCs for the linearized NS equations, which include
terms with second-order spatial derivatives. These kind of BCs
for the NS equations are discussed 
in \cite{CC1989,Dei1964,Maxwell,Sr1969}. It seems that the error
estimates for the general systems are not straightforward. This
would be a subject of a future study.

At this point, we would like to briefly review some related works. 
Considering the BCs for the moment system, most work is about the
numerical simulations
\cite{GuEmerson,GuptaTorrilhon,JulianTorrilhon,TorrilhonStruchtrup} 
while the theoretical results are rare. It was shown in
\cite{CaiFanLiQiao,CaiTorrilhon} that the number of derived BCs
for the moment system equals to the number of positive characteristic
speeds. In \cite{ZhaoYongBC}, the authors discuss the BCs for a
linearized four moments system and show the existence of relaxation 
limit. Our paper deals with the general moment system and aims at 
analyzing the boundary-layer behaviors. As for BCs of the NS 
equations, the slip BCs have been extensively studied both in
theory and experiment \cite{CC1989,Tropea2007}. The representative
work \cite{Coron} derived the slip BCs for steady flows with 
stationary boundaries. For time-dependent problems, the work
\cite{Aoki2017} presents a method to derive the slip BCs with
explicit values of the slip coefficients. This method is based
on the analysis of the Knudsen layer in \cite{Sone2007}. Note 
that the slip BCs derived in \cite{Aoki2017} are similar to our
results except for the terms with second-order derivatives. 

The rest of the paper is organized as follows. As a preparation, 
Section \ref{sec:Back} is devoted to reviewing the necessary background
and introducing basic notations. In Section \ref{sec:3}, 
by considering the Couette flow, we illustrate our basic idea through
a simplified model. For the general linearized moment system, the 
asymptotic equations are derived in Section \ref{sec:AA}. Based on
this, the BCs for the NS equations are constructed in Section
\ref{sec:NSBC}. At last, the numerical tests are presented in
Section \ref{sec:ANE}.

%% file: Back.tex
For convenience of readers, we briefly review some basic equations. 
Firstly, we introduce the celebrated linearized Boltzmann equation
(LBE) with the Maxwell BC. Then we present the
linear moment system with Grad's ansatz and the Grad BCs.
At last, we discuss the linearized NS equations
with slip BCs. 

\subsection{LBE with Maxwell BC}
We consider the nondimensional LBE (cf. \cite{CC1989}) 
\begin{equation}\label{eq:Bol}
	\partial_t f + \+\xi \cdot \nabla_{\+x} f = 
	\mathcal{L}[f]/\varepsilon,
\end{equation}
where $f=f(t,\+x,\+\xi)$ denotes the distribution 
function with $t\in\bbR^+$ the time,
$\+x=(x_1,x_2,x_3)\in\Omega\subset\bbR^3$ the spatial coordinates
and $\+\xi=(\xi_1,\xi_2,\xi_3)\in\bbR^3$ the microscopic velocity.
The linearized collision operator $\mathcal{L}$ describes the
interaction between particles and $\varepsilon$ represents the
Knudsen number. In this paper, we assume the Boltzmann equation
is linearized around a Maxwellian at rest:
\begin{equation}\label{eq:mMM}
	\mathcal{M}(\+\xi) = \frac{1}{(2\pi)^{3/2}}
	\exp\left(-\frac{|\+\xi|^2}{2} \right),
\end{equation}
and $\mathcal{L}[f]$ is defined as
\begin{equation*}
	\mathcal{L}[f](t,\+x,\+\xi) = \mathcal{M}(\+\xi)
	\int_{\bbR^3}\!\!\int_{\bbS^{2}}\!
   	\!K[f/\mathcal{M}]\mathcal{M}(\+\xi_*)
	B(|\+\xi-\+\xi_*|,\+\Theta)\,\mathrm{d}\+\Theta
	\mathrm{d}\+\xi_*, 
\end{equation*}
where $B(\cdot,\cdot)$ is a differential cross-section depending on
the potential between gas molecules. The function $K[\cdot]$ 
is defined as
\begin{equation*}
K[\psi](\+\xi,\+\xi_*,\+\Theta)
= \psi(\+\xi_*') + \psi(\+\xi') - \psi(\+\xi_*)
-\psi(\+\xi),
\end{equation*}
where $\+\xi'$ and $\+\xi_*'$ are determined by $\+\xi,\+\xi_*$ and
$\+\Theta$ from the elastic collision process. 

Then the macroscopic variables can be defined as 
\begin{eqnarray}\label{eq:macro1}
\begin{aligned}
	&	\rho = \ang{f},\quad 
	u_i = \ang{\xi_if},\quad
	p_{ij} = \ang{\xi_i\xi_jf},\quad p=\frac{1}{3}\sum_{i=1}^3p_{ii},\\
	& \theta = p-\rho = \ang{\frac{|\+\xi|^2-3}{3}f},\quad 
	\sigma_{ij}=p_{ij}-p\delta_{ij},\quad
	q_i = \ang{\frac{|\+\xi|^2-5}{2}\xi_if},
\end{aligned}
\end{eqnarray}
where
\begin{eqnarray*}
\ang{\cdot} = \int_{\bbR^3}\!\!\cdot\,\mathrm{d}\+\xi.
\end{eqnarray*}
We call $\rho=\rho(t,\+x)$ the density, $\+u=\+u(t,\+x)=(u_1,u_2,u_3)$ the
macroscopic velocity, $\theta=\theta(t,\+x)$ the temperature,
$p_{ij}$ the pressure tensor, $p$ the pressure, $\sigma_{ij}$
the stress tensor and $\+q=(q_1,q_2,q_3)$ the heat flux.

At a simple boundary, i.e., there is no mass flow across it, one 
extensively used BC is the so-called Maxwell BC
\begin{equation}\label{eq:Maxwell}
f(t,\+x,\+\xi) = \left\{{\begin{array}{*{20}l}
f(t,\+x,\+\xi),\qquad\qquad\qquad\qquad\qquad\quad\quad
 	(\+\xi-\+u^w) \cdot \+n\geq 0,\\[3mm]
\chi f^w(t,\+x,\+\xi) + (1-\chi) f(t,\+x,\+\xi^*),
   	~\quad\quad (\+\xi-\+u^w) \cdot \+n < 0,
\end{array}}\right.
\end{equation}
where $\+n=\+n(t,\+x)=(n_1,n_2,n_3)$ is the normal vector pointing out the 
domain and $\chi\in[0,1]$ is the (tangential momentum) accommodation 
coefficient. When $\chi=0$, the BC turns to the specular-reflection
BC, with
\[\+\xi^* = \+\xi - 2[(\+\xi-\+u^w)\cdot \+n] \+n.\]
When $\chi=1$, the BC is called the diffuse-reflection BC, determined
by 
$$
f^w(t,\+x,\+\xi) = \mathcal{M}(\+\xi)\left(
			\rho^w + \+u^w\cdot\+\xi + 
			\theta^w\frac{|\+\xi|^2-3}{2}\right),
$$
where $\+u^w=\+u^w(t,\+x)$ and $\theta^w=\theta(t,\+x)$ are (macro) 
velocity and temperature of the wall. The density $\rho^w$ is chosen
such that the no mass flow condition holds at the wall:
\begin{equation*}
	\int_{\mathbb{R}^3} [(\+\xi-\+u^w)\cdot \+n] f\,\mathrm{d}\+\xi = 0.
\end{equation*}

For simplicity, here and hereafter we assume the boundary is
$\partial\Omega=\{x_2=0\}$ with $\+n=(0,-1,0)$, and 
\[\+u^w\cdot\+n=0.\]

\subsection{Moment system with BCs}
In Grad's framework \cite{Grad1949}, one can take moments on 
both sides of the LBE to obtain an infinite moment system
\begin{equation}\label{eq:wei1}
	\pd{}{t}\ang{f\phi_{\+\alpha}} + \sum_{d=1}^3\pd{}{x_d}
	\ang{\xi_d f\phi_{\+\alpha}}	= \frac{1}{\varepsilon}
	\ang{\mathcal{L}[f]\phi_{\+\alpha}},\ 
	\+\alpha=(\alpha_1,\alpha_2,\alpha_3)\in\bbN^3,
\end{equation}
where $\phi_{\+\alpha}=\phi_{\+\alpha}(\+\xi)$ is the orthonormal 
Hermite polynomial \cite{Grad1949N} with the weight 
function $\mathcal{M}$ given in \eqref{eq:mMM}. Namely,
\[
	\ang{\mathcal{M}\phi_{\+\alpha}\phi_{\+\beta}} =
	\delta_{\+\alpha,\+\beta}.
\]
If we let $\phi_{\+0}=1$, then the above relation will
\cite{Grad1949N} lead to $\phi_{\+e_i}=\xi_i$ and 
\begin{equation}\label{eq:cur}
\xi_d \phi_{\+\alpha} = \sqrt{\alpha_d}\phi_{\+\alpha-\+e_d}
+\sqrt{\alpha_d+1}\phi_{\+\alpha+\+e_d},\quad d=1,2,3,
\end{equation}
where $\+e_i\in\bbN^3$ is the $i$-th standard basis vector. 

Denote the moment variables by
\[w_{\+\alpha}=\ang{f\phi_{\+\alpha}}.\]
By definition, we can immediately relate the low-order moment 
variables with the macroscopic variables defined in \eqref{eq:macro1} as
\begin{eqnarray}\label{eq:macro-var}
	\begin{aligned} &
	\rho=w_{\+0},\quad u_i=w_{\+e_i},\quad \theta=\frac{\sqrt{2}}{3}
		\sum_{i=1}^3w_{2\+e_i},\\ &
	\sigma_{ij}=\sqrt{1+\delta_{ij}}w_{\+e_i+\+e_j}-\theta\delta_{ij},
	\quad q_i = \frac{1}{2}
	\sum_{j=1}^3\sqrt{(\+e_i+2\+e_j)!}w_{\+e_i+2\+e_j},
\end{aligned}
\end{eqnarray}
where $\+\alpha!=\alpha_1!\alpha_2!\alpha_3!$. 

To close the moment system, Grad's moment 
method considers the ansatz
\[f\in\mathrm{span}\{\mathcal{M}\phi_{\+\alpha},\ |\+\alpha|\leq M\}
\quad\Rightarrow\quad f=\mathcal{M}\sum_{|\+\alpha|\leq
M}w_{\+\alpha}\phi_{\+\alpha},\]
where $M$ is a given integer called the moment order and
$|\+\alpha|=\alpha_1+\alpha_2+\alpha_3$. 
Substituting the ansatz into \eqref{eq:wei1} yields  
Grad's $M$-th order moment system
\begin{equation}\label{eq:Grad}
	\pd{W}{t} + \sum_{d=1}^3 A_d\pd{W}{x_d} = -\frac{1}{\varepsilon}QW,
\end{equation}
where $W\in\bbR^N$ with $N=\#\{\+\alpha\in\bbN^3:\ |\alpha|\leq M\}$.
Following the notations in \cite{CaiFanLi}, the vector $W$ and the
matrices $A_d$, $Q$ can be written as
\begin{eqnarray}
	W[\mathcal{N}(\+\alpha)] &=& w_{\+\alpha},\notag\\[2mm]
	A_d[\mathcal{N}(\+\alpha),\mathcal{N}(\+\beta)] &=& \ang{\mathcal{M}
	\xi_d\phi_{\+\alpha}\phi_{\+\beta}} = \label{eq:defofAQ}
	\sqrt{\alpha_d}\delta_{\+\beta,\+\alpha-\+e_d}
+ \sqrt{\alpha_d+1}\delta_{\+\beta,\+\alpha+\+e_d},\\[2mm] \notag
	Q[\mathcal{N}(\+\alpha),\mathcal{N}(\+\beta)] 
	&=& -\ang{\mathcal{L}[\mathcal{M}
	\phi_{\+\beta}]\phi_{\+\alpha}}.
\end{eqnarray}
Here the notation $W[\mathcal{N}(\+\alpha)]$ represents
the $\mathcal{N}(\+\alpha)$-th
element of the vector $W\in\bbR^N$, where 
\[\mathcal{N}: \{
	\+\alpha\in\bbN^3:\ |\+\alpha|\leq M \}
\rightarrow \{1,2,...,N\}\]
is a one to one mapping given in Remark \ref{rem:11}. Analogously,
$A_d[\mathcal{N}(\+\alpha),\mathcal{N}(\+\beta)]$ 
represents the $\mathcal{N}(\+\alpha)$-th row and 
$\mathcal{N}(\+\beta)$-th column of the matrix
$A_d\in\bbR^{N\times N}$. 

\begin{remark}\label{rem:11}
In the definition of $\mathcal{N}$, the multi-index $\+\alpha\in\bbN^3$ 
with the second component $\alpha_2$ even is always ordered before
$\+\beta$ with $\beta_2$ odd. Then the indices are ordered by the
norm, and finally by the anti-lexicographic order. Namely,
	\begin{itemize}
		\item $\mathcal{N}(\+\alpha) < \mathcal{N}(\+\beta)$
			when $\alpha_2$ is even and $\beta_2$ is odd.
		\item When $\alpha_2$ and $\beta_2$ have the same parity, 
			\[\mathcal{N}(\+\alpha) < \mathcal{N}(\+\beta)
			\quad\Leftrightarrow\quad |\+\alpha|<|\+\beta|.\]
		\item When $\alpha_2,\ \beta_2$ have the same parity and
			$|\+\alpha|=|\+\beta|$,
			\[\mathcal{N}(\+\alpha) < \mathcal{N}(\+\beta)
			\quad\Leftrightarrow\quad \exists\ i,\ s.t.\ 
			\alpha_i > \beta_i,\ \text{and}\ 
			\alpha_j=\beta_j,\ \forall j< i.\]
	\end{itemize}
\end{remark}

Now we turn to consider the BCs for the moment system \eqref{eq:Grad}. 
Following Grad's idea \cite{Grad1949}, to ensure the continuity
of BCs when $\chi\rightarrow 0,$ one can test the Maxwell
BC \eqref{eq:Maxwell} by odd polynomials 
(with respect to $\xi_2$) to construct BCs for moment variables.
Noting that the moment system \eqref{eq:Grad} is symmetric hyperbolic
since $A_d$ is symmetric, the correct number of BCs should coincide
with the number of negative eigenvalues of the boundary matrix 
$\sum_{i=1}^3A_in_i = -A_2$. 

To find appropriate test polynomials, we define 
\[\mathbb{I}_e=\{\+\alpha\in\bbN^3:\ |\+\alpha|\leq M,\
\alpha_2 \text{\ even}\},\quad
\mathbb{I}_o=\{\+\alpha\in\bbN^3:\ |\+\alpha|\leq M,\
\alpha_2 \text{\ odd}\},\]
where we suppose $m=\#\mathbb{I}_e$ and $n=N-m=\#\mathbb{I}_o$. 
It was shown in \cite{Grad1949} that the boundary matrix
$-A_2$ has $n$ positive eigenvalues, $n$ negative eigenvalues 
and $m-n$ zero eigenvalues. Hence, Grad selected the test polynomials
$\phi_{\+\alpha}$ with $\+\alpha\in\mathbb{I}_o$. To show the
continuity of fluxes, we prefer to extracting $\xi_2$ from the 
odd polynomials and choosing the test polynomials as 
$\xi_2\phi_{\+\alpha}$ with $\+\alpha\in\mathbb{I}_e$ and
$|\+\alpha|\leq M-1$. According to this equivalent choice,
the Grad BCs read as (see Appendix \ref{app:G} for details)
\begin{equation}\label{eq:Gbc_0}
	E[\hat{\chi}S,M_o](W-b)=0,
\end{equation}
with $\+u\cdot\+n=0$ given by the no mass flow condition.
Here $E\in\bbR^{n\times m}$ is a projection matrix,
$S\in\bbR^{m\times m}$ symmetric positive definite
and $M_o\in\bbR^{m\times n}$ is of full column rank such that
\[A_2=\begin{bmatrix} 0 & M_o \\ M_o^T & 0 \end{bmatrix},\]
all with constant coefficients, details in Appendix \ref{app:G}.
Here 
\[\hat{\chi}=\displaystyle\frac{2\chi}{(2-\chi)\sqrt{2\pi}}\]
and $b[\mathcal{N}(\+\alpha)]=b_{\+\alpha},$
with $b_{\+0}=\rho^w,\ b_{\+e_i}=u_i^w,\ b_{2\+e_i}=\theta^w/\sqrt{2},$
otherwise $b_{\+\alpha}=0.$ 

Despite the extensive numerical applications, the Grad BCs have rare 
well-posed results other than the correct number of BCs. It's shown
\cite{Sarna2018} that the Grad BCs are unstable in the linearized case.
In Section \ref{sec:AA}, we will introduce a modification of the Grad BCs
which is maximal positive. We note that the maximal positive
BC has played a key role in the well-posedness of the linear symmetric
hyperbolic system \cite{1960Local,Osher1975,Rauch1985,Paolo1996}.

\subsection{NS equations with slip BCs}
The linearized NS equations can be obtained from the Chapman-Enskog
expansion of the LBE or moment equations \cite{Grad1958,CC1989}.
The dimensionless equations read as
\begin{align}\label{linearNS}
&\frac{\partial {\rho}}{\partial {t}} + \sum_{d=1}^3 
	\frac{\partial {u}_d}{\partial {x}_d} = 0, \nonumber \\[2mm]
& \frac{\partial {u}_i}{\partial {t}} + \frac{\partial {p}}
	{\partial {x}_i} = \varepsilon\sum_{d=1}^3\pd{}{x_d}\left(\mu
	\left(\pd{u_i}{x_d}+\pd{u_d}{x_i}-\frac{2}{3}\delta_{id}
	\nabla\cdot\+u\right)\right),\\[2mm]
&\frac{3}{2} \frac{\partial {\theta}}{\partial {t}} +
	\sum_{d=1}^3 \frac{\partial {u}_d}{\partial {x}_d} =
	\frac{5}{2}\varepsilon\sum_{d=1}^3\pd{}{x_d}\left(\lambda\pd{\theta}
	{x_d}\right), \nonumber
\end{align}
which can be regarded as linearization around the reference density
$\rho_0=1$, velocity $u_{i0}=0$ and temperature $\theta_0=1$. Here 
the viscosity $\mu$ and the thermal conductivity $\lambda$ are 
constants due to the linearized assumption. 

The NS equations are usually equipped with the no-slip BCs.
In the region $\Omega=\{x_2>0\}$, the no-slip BCs are
\begin{equation}
	u_i=u_i^w,\quad \theta=\theta^w,\quad \text{at\ }x_2=0.
\end{equation}
Due to the evidence of velocity slip and temperature jump in 
experiments (cf. \cite{Tropea2007} and references cited therein), 
the slip BCs are recommended for the NS equations
when the Knudsen number is relatively large. 
At the boundary $x_2=0$, the first-order slip BCs
\cite{Aoki2017} read as 
\begin{align*}
	u_2 &= 0, \\[2mm]
	u_i-u_i^w &= a_0\varepsilon\left(\pd{u_i}{x_2}+
	\pd{u_2}{x_i}\right) +
	a_1\varepsilon\pd{\theta}{x_i},\ i=1,3,\\[2mm]
	\theta-\theta^w &= a_2\varepsilon\pd{\theta}{x_2} 
	+ a_3\varepsilon\pd{u_2}{x_2},
\end{align*}
where $\varepsilon$ means the Knudsen number and the constants
$a_i$ are slip coefficients. 
The first-order slip BCs lose its accuracy 
when the Knudsen number becomes larger (cf. \cite{Guo2014} and
references therein). To extend the application range of the
NS equations, many second-order slip BCs
have been proposed \cite{Maxwell,Sr1969,Dei1964,CC1989}, 
where the slip coefficients have a significant difference in
different literatures \cite{Colin2005}. An example of 
the second-order slip BCs for the tangential
velocity is 
\begin{equation*}
	u_1-u_1^w = a_0\varepsilon\pd{u_1}{\+n} + a_4\varepsilon^2
	\pd{^2u_1}{\+n^2},
\end{equation*}
where $a_4$ is the second-order slip coefficient.

%% file: Couette.tex
\section{Couette flow: an illustrative example}
\label{sec:3}


\subsection{Simplified moment system}\label{sec3.1}
In this section, we illustrate the basic idea to analyze the
boundary layers and to construct BCs for the NS equations through a
simplified model. To this end, we consider the Couette flow and make
some additional technical assumptions:
\begin{assumption}[Simplified model]\label{asp3.1}
~
\begin{itemize}
	\item[(i)] The flows are driven by the motion of one plate at
		$\{x_2=0\}$. The velocity of the plate is in the
		$x_1$-direction, i.e., $\+u^w=(u^w(t),0,0)$. The gradient in the $x_2$-direction is dominant and terms $\partial/\partial x_1$, $\partial/\partial x_3$ are omitted. Besides, only the velocity $u_1$ in the $x_1$-direction does not equal to zero. 
	\item[(ii)] The collision term is described by the BGK model \cite{BGK}.
	\item[(iii)] The number of moments is even which means the constant $M$ is odd. 
	\item[(iv)] The initial data of the system are prescribed as zero. 
\end{itemize}
\end{assumption}
\begin{remark} Notice that (i) and (iv) are common assumptions for the Couette flow.
Under the standard framework, the flows are often driven by two
	parallel plates. Here, for simplicity, we only consider one plate while the derivation for two plates case is similar. The assumption (ii) is made since the
BGK collision term is relatively simple. We make the
assumption (iii) to avoid the characteristic boundaries for the moment
system which are difficult to deal with. 
\end{remark}
Assumption
\ref{asp3.1} is only used in Section \ref{sec:3} while the general case
without these assumptions is considered in Section \ref{sec:AA} and \ref{sec:NSBC}.
Thanks to Assumption \ref{asp3.1}, the equations for the moments
$w_{\+\alpha}$ with $\+\alpha=\+e_1+k\+e_2~(k=0,1,\cdots,M)$ can be decoupled from the whole moment system \eqref{eq:Grad} (see \cite{Lijun2017,Gu2009}).
For simplicity of notations, we denote $x=x_2$ and $w_k=w_k(t,x)=w_{\+e_1+k\+e_2}$ for $k=0,1,\cdots,M$ in this section. Then the equations for the moments 
$$
W_c=\begin{pmatrix}
W_e\\[2mm]
W_o
\end{pmatrix}\quad \text{with}\quad W_e=(w_{0},w_{2},\cdots,w_{M-1})^T,\quad W_{o}=(w_{1},w_{3},\cdots,w_{M})^T
$$
can be written as
\begin{equation}\label{1}
\partial_tW_c+A_c
\partial_xW_c = -\frac{1}{\varepsilon}Q_cW_c.
\end{equation}
Here the coefficient matrices are 
$$
A_c = 
\begin{pmatrix}
0 & M_c\\[3mm]
M_c^T & 0
\end{pmatrix},\quad
M_c = 
\begin{pmatrix}
~1 & ~~ & ~~ & ~~ & ~~~\\[2mm]
\sqrt{2} & \sqrt{3} & ~~ & ~~ & ~~~\\[2mm]
~~ & \sqrt{4} & \sqrt{5}~ & ~~ & ~~~\\ 
~~ & ~~ & \ddots & \ddots & ~~~\\[2mm] 
~~ & ~~ & ~~ & \sqrt{M-1} & \sqrt{M} 
\end{pmatrix},\quad Q_c= \text{diag}(0,\underbrace{1,\cdots,1}_{M}).
$$
\begin{remark}\label{rem3.1}
From \eqref{eq:macro-var}, we know that $w_0=w_{\+e_1}$ is the
	velocity $u_1$ and $w_1=w_{\+e_1+\+e_2}$ is the stress $\sigma_{12}$. In this section, we denote $u=u_1$ and $\sigma=\sigma_{12}$ for short.
\end{remark}
Next we consider the BCs and the initial data.
Due to (iv) in Assumption \ref{asp3.1}, the initial data are prescribed as
$$
W_c(0,x)=0.
$$ 
On the other hand, the BCs for $W_c$ can also be decoupled from \eqref{eq:Gbc_0} as 
\begin{equation}\label{1b}
B_cW_c(t,0) = b_c(t).
\end{equation}
Here the coefficient matrix $B_c=(\hat{\chi}S_c,M_c)$. 
Note that $S_c$ is one part of the matrix $S$ in \eqref{eq:Gbc_0}. Since we only use the fact that $S_c$ is symmetric positive definite, the specific expression of $S_c$ is omitted. More details about this decoupling procedure can be found in \cite{Lijun2017,Gu2009}.
The right-hand side term in \eqref{1b} is
$$
b_c(t)=\hat{\chi}S_c(u^w(t),\underbrace{0,\cdots,0}_{(M-1)/2})^T.
$$
The initial data and BCs are assumed to be compatible at $(t,x)=(0,0)$. Namely,
$$
B_cW_c(0,0) = b_c(0),
$$
which means $u^w(0)=0$.

As to the IBVP \eqref{1}-\eqref{1b}, we claim that the BCs \eqref{1b}
are strictly dissipative. According to the classical theory of IBVPs 
for hyperbolic systems \cite{BenzoniSerre}, the strictly dissipative 
condition ensures the well-posedness. Moreover, from
\cite{Yong2,ZhouYong2} we know that this condition also
guarantees the existence of zero relaxation limit. 

\begin{prop}
Suppose the coefficient $0<\chi\leq 1$. Then the equations \eqref{1} with BCs \eqref{1b} are well-posed. Moreover, the solution admits a zero relaxation limit as $\varepsilon$ goes to zero.
\end{prop}
To prove this proposition, we recall the definition
\begin{definition}[Strictly dissipative condition]
The BCs \eqref{1b} for symmetric hyperbolic system \eqref{1} are referred to as strictly dissipative, if there is a positive constant $c$ such that 
$$
y^TA_cy\leq -c|y|^2+c^{-1}|B_cy|^2
$$
for all $y \in \mathbb{R}^n$.

\end{definition}

\begin{proof}[Proof of Proposition 3.1] 
It suffices to check that the following symmetric matrix is positive definite:
\begin{align*}
B_c^TB_c-cA_c-c^2I
= \begin{pmatrix}
\vs{S}^2-c^2I & (\vs{S}-cI)M_c \\[2mm]
M_c^T(\vs{S}-cI) & M_c^TM_c -c^2 I
\end{pmatrix}.
\end{align*}
Here $\vs{S}=\hat{\chi}S_c$.
By a congruent transformation, we only need to discuss the positiveness of 
\begin{align*}
\begin{pmatrix}
\vs{S}^2-c^2I & 0 \\[2mm]
0 & K
\end{pmatrix}\quad \text{with}\quad K = M_c^TM_c -c^2 I - M_c^T(\vs{S}-cI) \left(\vs{S}^2-c^2I\right)^{-1}(\vs{S}-cI)M_c.
\end{align*}
For sufficiently small $c>0$, we notice that $(\vs{S}^2-c^2I)^{-1}=\vs{S}^{-1}\vs{S}^{-1}+O(c^2)$ and thereby 
$$
K = 2cM_c^T\vs{S}^{-1}M_c + O(c^2).
$$
For $\hat{\chi}>0$, the matrix $\vs{S}$ is positive definite and
	thereby $K$ is positive definite for sufficiently small $c$.
	Consequently, we verify the strictly dissipative condition for the
	BCs.
\end{proof}

\subsection{Asymptotic analysis}
We need to analyze the boundary-layer behavior of \eqref{1} with BCs
\eqref{1b} for sufficiently small $\varepsilon$. To this end, we
recall the theory of hyperbolic relaxation system \cite{Yong2,ZhouYong} and consider the ansatz:
\begin{align}\label{3}
W_c(t,x) = \os{W}_c(t,x) + \vs{W}_c(t,\frac{x}{\sqrt{\varepsilon}}) 
	+ \ks{W}_c(t,\frac{x}{\varepsilon}).
\end{align}
Here $\os{W}_c$ is the outer solution representing the quantities far
away from the boundary, the others are boundary-layer corrections
which satisfy 
$$
\ks{W}_c(t,\infty) = \vs{W}_c(t,\infty) = 0.
$$
We expand these three variables as 
$$
\os{W}_c = \sum_{j=0}^{\infty} (\sqrt{\varepsilon})^j ~ \os{W}_c^{(j)},\quad \vs{W}_c = \sum_{j=0}^{\infty} (\sqrt{\varepsilon})^j ~ \vs{W}_c^{(j)},\quad \ks{W}_c = \sum_{j=0}^{\infty} (\sqrt{\varepsilon})^j ~ \ks{W}_c^{(j)}.
$$
For each $j$, the component-wise form for the coefficient reads as 
\begin{align*}
\os{W}_c^{(j)} = \left(\os{u}^{(j)},\os{w}_2^{(j)}\cdots,\os{w}_{M-1}^{(j)},\os{\sigma}^{(j)},\os{w}_3^{(j)}\cdots,\os{w}_M^{(j)}\right)^T,\\[2mm]
\vs{W}_c^{(j)} = \left(\vs{u}^{(j)},\vs{w}_2^{(j)}\cdots,\vs{w}_{M-1}^{(j)},\vs{\sigma}^{(j)},\vs{w}_3^{(j)}\cdots,\vs{w}_M^{(j)}\right)^T,\\[2mm]
\ks{W}_c^{(j)} = \left(\ks{u}^{(j)},\ks{w}_2^{(j)}\cdots,\ks{w}_{M-1}^{(j)},\ks{\sigma}^{(j)},\ks{w}_3^{(j)}\cdots,\ks{w}_M^{(j)}\right)^T.
\end{align*}
Notice that we have used the notations $u=w_0$ and $\sigma=w_1$ (see Remark \ref{rem3.1}). 


\subsubsection{Equations for coefficients}
\textbf{(1) Outer solution}

The outer solution $\os{W}_c$ should approximately satisfy the equation \eqref{1}.
We substitute $\os{W}_c$ into \eqref{1} and require the system to be satisfied up to $O(\varepsilon^{1/2})$. By comparing coefficients of each order of $\varepsilon$, we obtain
\begin{eqnarray*}
\left\{
\begin{array}{l} 
O(\varepsilon^{-1}):\quad ~~\os{\sigma}^{(0)}=0,\quad \os{w}_k^{(0)}=0,~(k=2,3,\cdots,M),\\[4mm]
O(\varepsilon^{-1/2}):\quad \os{\sigma}^{(1)}=0,\quad \os{w}_k^{(1)}=0,~(k=2,3,\cdots,M),\\[4mm] 
O(\varepsilon^0):\quad ~~~~\partial_t\os{u}^{(0)} = 0, \quad \os{\sigma}^{(2)}=-\partial_x\os{u}^{(0)},\quad \os{w}_k^{(2)}=0,~(k=2,3,\cdots,M),\\[4mm] 
O(\varepsilon^{1/2}):\quad ~~\partial_t\os{u}^{(1)} = 0, \quad 
\os{\sigma}^{(3)}=-\partial_x\os{u}^{(1)}, \quad \os{w}_k^{(3)}=0,~(k=2,3,\cdots,M).
\end{array}
\right.
\end{eqnarray*}
Besides, we require the first equation in \eqref{1} to be satisfied up to $O(\varepsilon)$ which gives 
$$
O(\varepsilon):\quad \partial_t\os{u}^{(2)} + \partial_x\os{\sigma}^{(2)} = 0.
$$
The initial data of the outer solution $\os{W}_c(0,x)$ should be given
according to the initial data of \eqref{1}, which equal to zero by
Assumption \ref{asp3.1} (iv). Then it is not difficult to see from the
above equations that $\os{W}_c(t,x)\equiv 0$.

\noindent \textbf{(2) Viscous layer solution}

Similarly, we require the boundary-layer correction term $\vs{W}_c$ to
satisfy the system \eqref{1} up to $O(\varepsilon^{1/2})$. Let
$y=x/\sqrt{\varepsilon}$. 
Comparing coefficients of each order of $\varepsilon$ yields
\begin{eqnarray}\label{eq:viscous-1}
\left\{
\begin{array}{l} 
O(\varepsilon^{-1}):\quad ~~\vs{\sigma}^{(0)}=0,\quad \vs{w}_k^{(0)}=0,~(k=2,3,\cdots,M), \\[4mm]
O(\varepsilon^{-1/2}):\quad \vs{\sigma}^{(1)} = - \partial_y\vs{u}^{(0)},\quad \vs{w}_k^{(1)}=0,~(k=2,3,\cdots,M),\\[4mm] 
O(\varepsilon^0):\quad \quad ~\partial_t\vs{u}^{(0)} + \partial_y\vs{\sigma}^{(1)} = 0, \quad \vs{\sigma}^{(2)}=-\partial_y\vs{u}^{(1)},\quad \vs{w}_{2}^{(2)}=-\sqrt{2}\partial_y \vs{\sigma}^{(1)}, \\[4mm]
\qquad\qquad\quad~~ \vs{w}_k^{(2)}=0,~(k=3,4,\cdots,M),\\[4mm] 
O(\varepsilon^{1/2}):\quad~~ \partial_t\vs{u}^{(1)} + \partial_y\vs{\sigma}^{(2)} = 0,\quad \vs{\sigma}^{(3)}=-\partial_t\vs{\sigma}^{(1)}-\partial_y\vs{u}^{(2)}-\sqrt{2}\partial_y\vs{w}_2^{(2)},\\[4mm] 
\qquad\qquad\quad~~ \vs{w}_{2}^{(3)}=-\sqrt{2}\partial_y \vs{\sigma}^{(2)},\quad \vs{w}_{3}^{(3)}=-\sqrt{3}\partial_y \vs{w}_2^{(2)}, \quad \vs{w}_k^{(3)}=0,~(k=4,5,\cdots,M).
\end{array}
\right.
\end{eqnarray}
Moreover, we require the first equation in \eqref{1} to be satisfied up to $O(\varepsilon)$ which gives 
$$
O(\varepsilon):\quad \partial_t\vs{u}^{(2)} + \partial_y\vs{\sigma}^{(3)} = 0.
$$

Now we show how to solve the viscous layer solutions from the above equations. Firstly, we solve two parabolic equations 
\begin{equation}\label{eq:vis1}
	\partial_t\vs{u}^{(0)} - \partial_{yy}\vs{u}^{(0)} = 0,\qquad
\partial_t\vs{u}^{(1)} - \partial_{yy}\vs{u}^{(1)} = 0
\end{equation}
to obtain $\vs{u}^{(0)}$ and $\vs{u}^{(1)}$. Then we can obtain $\vs{\sigma}^{(1)}$, $\vs{\sigma}^{(2)}$ and $\vs{w}_2^{(2)}$ by algebraic relations in \eqref{eq:viscous-1}. Having these, we solve a parabolic equation 
\begin{equation}\label{eq:vis2}
\partial_t\vs{u}^{(2)} - \partial_{yy}\vs{u}^{(2)} = \partial_{ty}\vs{\sigma}^{(1)}+\sqrt{2}\partial_{yy}\vs{w}_2^{(2)}
\end{equation}
to obtain $\vs{u}^{(2)}$. At last, by certain algebraic equations in \eqref{eq:viscous-1} we can determine $\vs{\sigma}^{(3)}$, $\vs{w}_{2}^{(3)}$ and $\vs{w}_{3}^{(3)}$. Clearly, the equations \eqref{eq:vis1}-\eqref{eq:vis2} need BCs for $\vs{u}^{(0)}$, $\vs{u}^{(1)}$ and $\vs{u}^{(2)}$. 

\noindent\textbf{(3) Knudsen layer solution} 

Let $z=x/\varepsilon$. We require $\ks{W}_c$ to satisfy the
equation \eqref{1} up to $O(\varepsilon^{1/2})$, which yields 
\begin{eqnarray}\label{eq:knudsen-1}
\begin{array}{l}
O(\varepsilon^{(j-2)/2}):\\[2mm]
(j=0,1,2,3)~
\end{array}
\left\{\begin{array}{l}
\partial_z\ks{\sigma}^{(j)} = -\partial_t\ks{u}^{(j-2)},\quad \partial_z\ks{u}^{(j)}+\sqrt{2}\partial_z\ks{w}_2^{(j)}=-\ks{\sigma}^{(j)}-\partial_t\ks{\sigma}^{(j-2)}, \\[5mm]
\ks{A}_c
\partial_z\begin{pmatrix}
\ks{w}_e^{(j)} \\[2mm]
\ks{w}_o^{(j)}
\end{pmatrix}=-\begin{pmatrix}
\ks{w}_e^{(j)} \\[2mm]
\ks{w}_o^{(j)}
\end{pmatrix}
-\partial_t
\begin{pmatrix}
\ks{w}_e^{(j-2)} \\[2mm]
\ks{w}_o^{(j-2)}
\end{pmatrix}.
\end{array}\right.
\end{eqnarray} 
Here the notations are defined by
$$
\ks{w}_e^{(j)} = 
\begin{pmatrix}
\ks{w}_2^{(j)} \\[2mm]
\ks{w}_4^{(j)} \\[2mm]
\vdots \\[2mm]
\ks{w}_{M-1}^{(j)}
\end{pmatrix},\quad 
\ks{w}_o^{(j)} = 
\begin{pmatrix}
\ks{w}_3^{(j)} \\[2mm]
\ks{w}_5^{(j)} \\[2mm]
\vdots \\[2mm]
\ks{w}_M^{(j)}
\end{pmatrix},\quad
\ks{A}_c=
\begin{pmatrix}
0 & \ks{M}_c \\[2mm]
\ks{M}_c^T & 0
\end{pmatrix},\quad
\ks{M}_c = \begin{pmatrix}
\sqrt{3} & ~~ & ~~ & ~~~\\[2mm]
\sqrt{4} & \sqrt{5}~ & ~~ & ~~~\\ 
~~ & \ddots & \ddots & ~~~\\[2mm] 
~~ & ~~ & \sqrt{M-1} & \sqrt{M} 
\end{pmatrix}.
$$

Next we show the procedure to obtain Knudsen layer solutions from the above equations. For $j=0,1$, we need to solve the ODE systems
\begin{equation}\label{eq:knu1}
\ks{A}_c
\partial_z\begin{pmatrix}
\ks{w}_e^{(0)} \\[2mm]
\ks{w}_o^{(0)}
\end{pmatrix}=-\begin{pmatrix}
\ks{w}_e^{(0)} \\[2mm]
\ks{w}_o^{(0)}
\end{pmatrix},\qquad
\ks{A}_c
\partial_z\begin{pmatrix}
\ks{w}_e^{(1)} \\[2mm]
\ks{w}_o^{(1)}
\end{pmatrix}=-\begin{pmatrix}
\ks{w}_e^{(1)} \\[2mm]
\ks{w}_o^{(1)}
\end{pmatrix}.
\end{equation}
Then it is easy to see that 
\begin{equation}\label{eq:knu-sig-u}
	\ks{\sigma}^{(j)}=0,\quad \ks{u}^{(j)}=-\sqrt{2}\ks{w}_2^{(j)},\quad j=0,1.
\end{equation}
Having these, we proceed to solve 
\begin{equation}\label{eq:knu2}
\ks{A}_c
\partial_z\begin{pmatrix}
\ks{w}_e^{(2)} \\[2mm]
\ks{w}_o^{(2)}
\end{pmatrix}=-\begin{pmatrix}
\ks{w}_e^{(2)} \\[2mm]
\ks{w}_o^{(2)}
\end{pmatrix}
-\partial_t
\begin{pmatrix}
\ks{w}_e^{(0)} \\[2mm]
\ks{w}_o^{(0)}
\end{pmatrix},
\qquad
\ks{A}_c
\partial_z\begin{pmatrix}
\ks{w}_e^{(3)} \\[2mm]
\ks{w}_o^{(3)}
\end{pmatrix}=-\begin{pmatrix}
\ks{w}_e^{(3)} \\[2mm]
\ks{w}_o^{(3)}
\end{pmatrix}
-\partial_t
\begin{pmatrix}
\ks{w}_e^{(1)} \\[2mm]
\ks{w}_o^{(1)}
\end{pmatrix}.
\end{equation}
At last, $\ks{\sigma}^{(2)}$, $\ks{u}^{(2)}$ and $\ks{\sigma}^{(3)}$, $\ks{u}^{(3)}$ are determined by relations 
\begin{equation}\label{eq:knu-sig-u-2}
	\partial_z\ks{\sigma}^{(2)} = -\partial_t\ks{u}^{(0)},\quad \partial_z\ks{u}^{(2)}=-\sqrt{2}\partial_z\ks{w}_2^{(2)}-\ks{\sigma}^{(2)} 
\end{equation}
and 
\begin{equation}\label{eq:knu-sig-u-3}
\partial_z\ks{\sigma}^{(3)} = -\partial_t\ks{u}^{(1)},\quad \partial_z\ks{u}^{(3)}=-\sqrt{2}\partial_z\ks{w}_2^{(3)}-\ks{\sigma}^{(3)}.	 
\end{equation}

To obtain bounded solutions to the ODE systems \eqref{eq:knu1} and \eqref{eq:knu2}, we state
\begin{prop}\label{prop11}
\noindent There is an orthogonal matrix $R$ satisfying  
$$
\ks{A}_cR=R\begin{pmatrix}
\ks{\Lambda}_+ & 0 \\[2mm]
0 & -\ks{\Lambda}_+
\end{pmatrix},\qquad R=\begin{pmatrix}
R_e & R_e \\[2mm]
R_o & -R_o
\end{pmatrix}.
$$ 
Here $R_o$ and $R_e$ are $(\frac{M-1}{2}\times \frac{M-1}{2})$-invertible matrices, $\ks{\Lambda}_+$ is a positive diagonal matrix.
\end{prop}
\noindent Thanks to this proposition, we can express the moments by characteristic variables
$$
\begin{pmatrix}
\ks{w}_e^{(k)} \\[2mm]
\ks{w}_o^{(k)}
\end{pmatrix}=
\begin{pmatrix}
R_e \\[3mm]
-R_o
\end{pmatrix}
\ks{w}_{+}^{(k)} + 
\begin{pmatrix}
R_e \\[3mm]
R_o
\end{pmatrix}
\ks{w}_{-}^{(k)}.
$$
For bounded solutions, we require $\ks{w}_{+}^{(k)} \equiv 0$ and get 
\begin{equation}\label{eq:knu3}
	\ks{w}_e^{(k)} = R_e\ks{w}_{-}^{(k)}, \quad \ks{w}_o^{(k)} = R_o\ks{w}_{-}^{(k)}.
\end{equation}
The value of $\ks{w}_{-}^{(k)}$ should be determined from the BCs.

\subsubsection{Boundary conditions}\label{sec:3.2.2}

Substituting the asymptotic solution into the BCs \eqref{1b} and matching each order of $\varepsilon$ yield
$$
O(\varepsilon^{j/2})~(j=0,1,2):\quad
(\hat{\chi}S_c, M_c)[\vs{W}_c^{(j)}(t,0)+\ks{W}_c^{(j)}(t,0)] = b_c^{(j)}
$$
with 
$$
b_c^{(0)}=\hat{\chi}S_c(u^w,0,\cdots,0)^T, \quad b_c^{(1)} = b_c^{(2)} = 0.
$$ 
For further discussions, we denote 
$$
M_c = \begin{pmatrix}
~1 & 0 \\[2mm]
~\ks{g} & \ks{M}_c
\end{pmatrix},\quad \ks{g} = (\sqrt{2},\underbrace{0,\cdots,0}_{\frac{M-3}{2}})^T,\quad
\vs{w}_e^{(j)} = (\vs{w}_2^{(j)},\cdots,\vs{w}_{M-1}^{(j)})^T,\quad
\vs{w}_o^{(j)} = (\vs{w}_3^{(j)},\cdots,\vs{w}_{M}^{(j)})^T.
$$ 
Then the above equations can be rewritten as
\begin{equation}\label{eq:bc-1}
O(\varepsilon^{j/2})~(j=0,1,2):\quad
\hat{\chi}S_c
\begin{pmatrix}
\vs{u}^{(j)}+\ks{u}^{(j)} \\[3mm]
\vs{w}_e^{(j)}+\ks{w}_e^{(j)}
\end{pmatrix}
+\begin{pmatrix}
~1 & 0 \\[3mm]
~\ks{g} & \ks{M}_c
\end{pmatrix}
\begin{pmatrix}
\vs{\sigma}^{(j)}+\ks{\sigma}^{(j)} \\[3mm]
\vs{w}_o^{(j)}+\ks{w}_o^{(j)}
\end{pmatrix}=b_c^{(j)}.	
\end{equation}
  
\noindent \textbf{(1) Order $\varepsilon^0$:}

Recall that $\vs{\sigma}^{(0)}=0$ and $\ks{\sigma}^{(0)}=0$. From \eqref{eq:knu-sig-u}, we know that $\ks{u}^{(0)}=-\sqrt{2}\ks{w}_2^{(0)}=-\ks{g}^T\ks{w}_{e}^{(0)}$. Then it follows from \eqref{eq:knu3} that  
$$
\hat{\chi}S_c
\begin{pmatrix}
\vs{u}^{(0)}-\ks{g}^TR_e\ks{w}_{-}^{(0)} \\[2mm]
R_e\ks{w}_{-}^{(0)}
\end{pmatrix}
+\begin{pmatrix}
~1 & 0 \\[3mm]
~\ks{g} & \ks{M}_c
\end{pmatrix}
\begin{pmatrix}
0 \\[3mm]
R_o\ks{w}_{-}^{(0)}
\end{pmatrix}=b_c^{(0)}.
$$  
The unknowns are $\vs{u}^{(0)}$ and $\ks{w}_{-}^{(0)}$. Thus we rewrite the above equation as
\begin{equation}\label{eq:bc-order0}
\left[\hat{\chi}S_c\begin{pmatrix}
~1 & -\ks{g}^TR_e\\[2mm]
~0 & R_e
\end{pmatrix}+\begin{pmatrix}
0 & 0\\[2mm]
0 & \ks{M}_cR_o
\end{pmatrix}\right]
\begin{pmatrix}
\vs{u}^{(0)} \\[2mm]
\ks{w}_{-}^{(0)}
\end{pmatrix}=\hat{\chi}S_c\begin{pmatrix}
u^w \\[2mm]
0
\end{pmatrix}.
\end{equation}
To show the solvability of this equation, we state
\begin{lemma}\label{lemma:3-1}
For $\hat{\chi}>0$, the matrix 
$$
H_M(\chi): = 
\hat{\chi}S_c\begin{pmatrix}
~1 & -\ks{g}^TR_e\\[2mm]
~0 & R_e
\end{pmatrix}+\begin{pmatrix}
0 & 0\\[2mm]
0 & \ks{M}_cR_o
\end{pmatrix}
$$
is invertible. 
\end{lemma}
\begin{proof}
Proposition \ref{prop11} implies that $\ks{M}_cR_o = R_e\ks{\Lambda}_+$ and $2R_eR_e^T=I$. Then we multiply $H_M(\chi)$ with 
$$
\begin{pmatrix}
1 & \ks{g}^T\\[2mm]
0 & 2R_e^T
\end{pmatrix}
$$
from the right to obtain 
$$
\left[\hat{\chi}S_c\begin{pmatrix}
~1 & -\ks{g}^TR_e\\[2mm]
~0 & R_e
\end{pmatrix}+\begin{pmatrix}
0 & 0\\[2mm]
0 & R_e\ks{\Lambda}_+
\end{pmatrix}\right]
\begin{pmatrix}
1 & \ks{g}^T\\[2mm]
0 & 2R_e^T
\end{pmatrix} = \hat{\chi}S_c+\begin{pmatrix}
0 & 0\\[2mm]
0 & 2R_e\ks{\Lambda}_+R_e^T
\end{pmatrix}.
$$
Clearly, the matrix on the right-hand side is positive definite since $S_c$ is positive definite and $\hat{\chi}>0$. Consequently, we get the invertibility stated in lemma.
\end{proof}

Thanks to this lemma, the unique solution to \eqref{eq:bc-order0} is 
\begin{equation}\label{eq:bc-order0-1}
	\vs{u}^{(0)}(t,0) = u^w(t),\quad \ks{w}_{-}^{(0)}(0,t) = 0.
\end{equation}
It follows from \eqref{eq:knu-sig-u} and \eqref{eq:knu3} that $\ks{W}_c^{(0)}\equiv 0$.

\noindent\textbf{(2) Order $\varepsilon^{1/2}$:}

Similarly, we can derive the BCs for $j=1$. 
Recall that $\vs{\sigma}^{(1)} = - \partial_y\vs{u}^{(0)}$, $\ks{\sigma}^{(1)}=0$ and $\ks{u}^{(1)}=-\sqrt{2}\ks{w}_2^{(1)}=-\ks{g}^T\ks{w}_{e}^{(1)}$. Substituting into \eqref{eq:bc-1} and using \eqref{eq:knu3} yields
\begin{equation*}
\left[\hat{\chi}S_c\begin{pmatrix}
~1 & -\ks{g}^TR_e\\[2mm]
~0 & R_e
\end{pmatrix}+\begin{pmatrix}
0 & 0\\[2mm]
0 & \ks{M}_cR_o
\end{pmatrix}\right]
\begin{pmatrix}
\vs{u}^{(1)} \\[2mm]
\ks{w}_{-}^{(1)}
\end{pmatrix}=
\partial_y\vs{u}^{(0)}\begin{pmatrix}
 ~1~ \\[1mm]
 ~\ks{g}~
\end{pmatrix}.	
\end{equation*}
Due to Lemma \ref{lemma:3-1}, this equation is uniquely solvable and we have 
\begin{equation}\label{eq:bc-order1}
	\vs{u}{1}(t,0) = K_M(\chi) \partial_y\vs{u}{0}(t,0)
	\quad \text{with} \quad
	K_M(\chi):= (1,~0)H_M^{-1}\begin{pmatrix}
 ~1~ \\[1mm]
 ~\ks{g}~
\end{pmatrix}.
\end{equation}
By resorting to the proof of Lemma \eqref{lemma:3-1}, we find that 
\begin{align*}
K_M(\chi)= &~ (1,~0)\begin{pmatrix}
1 & \ks{g}^T\\[2mm]
0 & 2R_e^T
\end{pmatrix}
\left[\hat{\chi}S_c+\begin{pmatrix}
0 & 0\\[2mm]
0 & 2R_e\ks{\Lambda}_+R_e^T
\end{pmatrix}\right]^{-1}
\begin{pmatrix}
 ~1~ \\[1mm]
 ~\ks{g}~
\end{pmatrix} \\[2mm]
=&~(1, ~\ks{g}^T) 
\left[\hat{\chi}S_c+\begin{pmatrix}
0 & 0\\[2mm]
0 & 2R_e\ks{\Lambda}_+R_e^T
\end{pmatrix}\right]^{-1}
\begin{pmatrix}
 ~1~ \\[1mm]
 ~\ks{g}~
\end{pmatrix}>0
\end{align*}
since $S_c, \ks{\Lambda}_+$ are positive definite and $\hat{\chi}>0$.

\noindent \textbf{(3) Order $\varepsilon^1$:}

At last, we derive the BCs for $\vs{u}^{(2)}$. Since $\ks{u}^{(0)}=0$, it follows from \eqref{eq:knu-sig-u-2} that $\ks{\sigma}^{(2)} = 0$ and $\ks{u}^{(2)}=-\sqrt{2}\ks{w}_2^{(2)}=-\ks{g}^T\ks{w}_{e}^{(2)}$. Besides, recall that $\vs{\sigma}^{(2)}=-\partial_y\vs{u}^{(1)}$, $\vs{w}_{2}^{(2)}=-\sqrt{2}\partial_y \vs{\sigma}^{(1)}=\sqrt{2}\partial_{yy} \vs{u}^{(0)}$. Due to the equation \eqref{eq:vis1} and \eqref{eq:bc-order0-1}, it follows that 
\begin{equation}\label{eq:couette2edderi}
\vs{w}_{2}^{(2)}(t,0) = \sqrt{2}\partial_{yy} \vs{u}^{(0)} = \sqrt{2}\partial_{t} \vs{u}^{(0)}(t,0) = \sqrt{2}\partial_{t} u^{w}(t).	
\end{equation}
Moreover, by the definition of $\vs{w}_{e}^{(2)}$ and $\ks{g}$ we can write $\vs{w}_{e}^{(2)}=\ks{g}\partial_{t} u_1^{w}(t)$. 
Substituting these relations into \eqref{eq:bc-1} and using \eqref{eq:knu3} yields
\begin{equation*}
\left[\hat{\chi}S_c\begin{pmatrix}
~1 & -\ks{g}^TR_e\\[2mm]
~0 & R_e
\end{pmatrix}+\begin{pmatrix}
0 & 0\\[2mm]
0 & \ks{M}_cR_o
\end{pmatrix}\right]
\begin{pmatrix}
\vs{u}^{(2)} \\[2mm]
\ks{w}_{-}^{(2)}
\end{pmatrix}=
\begin{pmatrix}
 ~1~ \\[1mm]
 ~\ks{g}~
\end{pmatrix}\partial_y\vs{u}^{(1)}
-\hat{\chi}S_c\begin{pmatrix}
 ~0~ \\[1mm]
 ~\ks{g}~
\end{pmatrix}\partial_{t} u^{w}(t).	
\end{equation*}
Thanks to Lemma \ref{lemma:3-1}, we can solve the above algebraic equation to obtain 
\begin{equation}\label{eq:bc-order2-1}
	\vs{u}{2}(t,0) = K_M(\chi) \partial_y\vs{u}{1}(t,0) + J_M(\chi)
	\partial_t u^w(t).
\end{equation}
Here the constant $K_M(\chi)$ is defined in \eqref{eq:bc-order1} and 
$$
J_M(\chi) := -(1,~0)H_M(\chi)^{-1}\hat{\chi}S_c\begin{pmatrix}
0 \\
\ks{g}
\end{pmatrix}.
$$ 

\subsubsection{Validity}
Construct an approximate solution to \eqref{1}:
$$
W_c^{app}=\os{W}_c^{app}+\vs{W}_c^{app}+\ks{W}_c^{app}.
$$
Here $\os{W}_c^{app}=0$ and  
$$
\vs{W}_c^{app}
=\vs{W}_c^{(0)}+\sqrt{\varepsilon}\vs{W}_c^{(1)}+\varepsilon \vs{W}_c^{(2)}+\varepsilon^{2/3}
\left(0,\vs{w}_2^{(3)},\cdots,\vs{w}_{M-1}^{(3)}, \vs{\sigma}^{(3)},\vs{w}_{3}^{(3)}\cdots,\vs{w}_{M}^{(3)}\right)^T
$$
where the coefficients are determined by the equations \eqref{eq:viscous-1} with BCs \eqref{eq:bc-order0-1}-\eqref{eq:bc-order2-1}. 
Moreover, the term $\ks{W}_c^{app}$ is constructed by
$$
\ks{W}_c^{app}
=\ks{W}_c^{(0)}+\sqrt{\varepsilon}\ks{W}_c^{(1)}+\varepsilon \ks{W}_c^{(2)}+\varepsilon^{2/3}\ks{W}_c^{(3)},
$$
where the coefficients are determined by \eqref{eq:knu1}-\eqref{eq:knu-sig-u-3}. The above approximate solutions are truncated from $\vs{W}_c$ and $\ks{W}_c$. 
Notice that the term $\vs{u}^{(3)}$ is not involved in the expression of $\vs{W}_c^{app}$ since it can not be determined from \eqref{eq:viscous-1}.

Substituting the approximate solution $W_c^{app}$ into the equation \eqref{1}, we have 
\begin{align*}
\partial_tW_c^{app}+A_c\partial_xW_c^{app}=-\frac{1}{\varepsilon}Q_cW_c^{app}+
\begin{pmatrix}
0\\[1mm]
\varepsilon \vs{R}
\end{pmatrix} + \varepsilon \ks{R}
\end{align*}
by our construction \eqref{eq:viscous-1} and \eqref{eq:knudsen-1}.
Here the residual $\vs{R}=\vs{R}(x/\sqrt{\varepsilon})$ has $M$ components and $\ks{R}=\ks{R}(x/\varepsilon)$ has $M+1$ components. Moreover, there exists a constant $C$ independent of $\varepsilon$ such that 
$$
\int_{\mathbb{R}^+}\Big|\vs{R}(\frac{x}{\sqrt{\varepsilon}})\Big|^2 dx = \sqrt{\varepsilon} \int_{\mathbb{R}^+}\big|\vs{R}(y)\big|^2 dy \leq C \varepsilon^{1/2}
$$
and
$$
\int_{\mathbb{R}^+}\Big|\ks{R}(\frac{x}{\varepsilon})\Big|^2 dx = \varepsilon \int_{\mathbb{R}^+}\big|\ks{R}(z)\big|^2 dz \leq C \varepsilon.
$$ 
On the other hand, according to the discussion in subsection
\ref{sec:3.2.2}, we substitute $W_c^{app}$ into the BCs \eqref{1b} and
obtain
\begin{equation}
	B_cW_c^{app}(t,0) = \varepsilon^{3/2} h(t),
\end{equation}
where $h(t)$ satisfies $\int_0^Th(t)^2dt\leq C$. At last, the initial data of $W_c^{app}$ are given as zero. 

Next, we show the validity of the approximate solution $W_c^{app}$ by the following theorem:
\begin{theorem}
The approximate solution $W_c^{app}$ and the exact solution $W_c$ to the equation \eqref{1} satisfy the following error estimate for any $t\in[0,T]$:
$$
	\|(W_c^{app}-W_c)(t,\cdot)\|_{L^2(\mathbb{R}^+)} + \left(\int_0^T
	|W_c^{app}(t,0)-W_c(t,0)|^2 dt\right)^{1/2} \leq C(T) \varepsilon^{3/2}.
$$
\end{theorem}
\begin{proof}
Denote the error $E(t,x)=W_c^{app}(t,x)-W_c(t,x)$. It satisfies the IBVP:
\begin{eqnarray}\label{eq:errorequation}
\left\{\begin{array}{l}
\partial_t E + A_c\partial_x E = -\dfrac{1}{\varepsilon}Q_cE + \begin{pmatrix}
0\\[1mm]
\varepsilon \vs{R}
\end{pmatrix} + \varepsilon \ks{R},\\[3mm]
B_cE(t,0) = \varepsilon^{3/2} h(t),\\[3mm]
E(0,x) = 0.
\end{array}
\right.
\end{eqnarray} 
Multiplying $E^T(t,x)$ on the left side and integrating over $x\in [0,\infty)$ yields 
\begin{align*}
&\frac{d}{dt}\|E(t,\cdot)\|_{L^2(\mathbb{R}^+)}^2 - E^T(t,0)A_cE(t,0)  \\[2mm]
=& -\dfrac{2}{\varepsilon}\|E^{II}\|_{L^2(\mathbb{R}^+)}^2 + 2\varepsilon \int_{\mathbb{R}^+}(E^{II})^T\vs{R}dx + 2\varepsilon \int_{\mathbb{R}^+} E^T\ks{R} dx, \\[2mm] 
\leq &-\dfrac{2}{\varepsilon}\|E^{II}\|_{L^2(\mathbb{R}^+)}^2 +
	\varepsilon^3 \int_{\mathbb{R}^+}|\vs{R}|^2dx +
	\frac{1}{\varepsilon}\int_{\mathbb{R}^+}|E^{II}|^2dx +
	\varepsilon^2 \int_{\mathbb{R}^+} |\ks{R}|^2 dx +
	\int_{\mathbb{R}^+} |E|^2 dx\\[2mm]
\leq & ~C \varepsilon^3 + \|E(t,\cdot)\|_{L^2(\mathbb{R}^+)}^2.
\end{align*} 
Here the vector $E$ has $M+1$ components and $E^{II}$ represents for the last $M$ components of $E$. Since the BCs satisfy the strictly dissipative condition, it follows that 
$$
E^T(t,0)A_cE(t,0) \leq -c|E(t,0)|^2+c^{-1}\varepsilon^3|h(t)|^2.
$$
At last, by using the Gronwall's inequality, we obtain the estimate stated in the theorem.
\end{proof}

\subsection{Construction of BCs for NS equations}
In this subsection, we consider the simplified NS equation:
\begin{equation}\label{NSeq}
	\partial_tu_{ns}- \varepsilon\partial_{xx}u_{ns} =0.
\end{equation} 
\begin{remark}
In the original NS equations \eqref{linearNS}, the variables are the density $\rho$, the velocities $u_i$ and the temperature $\theta$. According to the discussion in Subsection \ref{sec3.1}, we only consider the velocity in the $x_1$-direction here (denote as $u_{ns}$). Moreover, by Assumption \ref{asp3.1} we can simplify \eqref{linearNS} and obtain \eqref{NSeq}.
\end{remark}

Our main goal is to construct BCs for the NS equation by resorting to the matching requirement stated in the introduction. 
To do this, we collect the coefficients of $\os{u}^{(j)}$ and $\vs{u}^{(j)}$ in \eqref{3} and compare them with the solution $u_{ns}$ to the NS equation. 
Since $\os{u}^{(j)}\equiv 0$, we construct an approximate solution to the NS equation by
$$
u^{a}_{ns}= \vs{u}{0}+\sqrt{\varepsilon}\vs{u}{1}+\varepsilon\vs{u}{2}.
$$
According to the equations of $\vs{u}^{(j)}$ in \eqref{eq:viscous-1}, we check that $u^{a}_{ns}$ satisfies
$$
\partial_tu^{a}_{ns} - \varepsilon\partial_{xx}u^{a}_{ns} = \varepsilon R^{ns}\left(\frac{x}{\sqrt{\varepsilon}}\right).
$$
Here the residual $R^{ns}=R^{ns}(x/\sqrt{\varepsilon})$ satisfies
$$
\int_{\mathbb{R}^+} |R^{ns}|^2\left(\frac{x}{\sqrt{\varepsilon}}\right) dx = \sqrt{\varepsilon}\int_{\mathbb{R}^+} |R^{ns}|^2(y) dy \leq C \sqrt{\varepsilon}. 
$$

On the other hand, we aim to provide a BC for the NS equation so that $u^{a}_{ns}$ approximately satisfies this BC as well. To this end, we construct 
\begin{equation}\label{eq:NSconstructedBC}
u_{ns}(t,0)-\varepsilon K_M(\chi)\partial_x u_{ns}(t,0) = u^w(t)+
	\varepsilon J_M(\chi)\partial_tu^w(t).
\end{equation}
Substituting $u^{a}_{ns}$ into this condition yields 
\begin{align*}
&~u^{a}_{ns}(t,0)-\varepsilon K_M(\chi)\partial_x u^{a}_{ns}(t,0) \\[1mm]
	=&~\vs{u}{0}(t,0) + \sqrt{\varepsilon} \vs{u}{1}(t,0) +
	\varepsilon \vs{u}{2}(t,0) - \varepsilon K_M(\chi)
	\left(\frac{1}{\sqrt{\varepsilon}}\partial_y\vs{u}{0}(t,0) +
	\partial_y\vs{u}{1}(t,0) \right)+ \varepsilon^{3/2}b_{ns}(t) \\[1mm]
=&~u^w(t)+ \varepsilon J_M(\chi)\partial_tu^w(t) + \varepsilon^{3/2}b_{ns}(t),
\end{align*}
where $\displaystyle\int_0^Tb_{ns}^2(t)dt \leq C$.
In order to show that the constructed BCs \eqref{eq:NSconstructedBC} 
fulfill the matching requirement, we state
\begin{theorem}
The solution $u_{ns}$ to the NS equation \eqref{NSeq} and the approximate solution $u^{a}_{ns}$ satisfy the following error estimate for any $t\in[0,T]$:
$$
\|(u_{ns}-u^{a}_{ns})(t,\cdot)\|_{L^2(\mathbb{R}^+)} +
	\left(\int_0^T|u_{ns}(t,0)-u^{a}_{ns}(t,0)|^2dt\right)^{1/2}
	\leq C(T) \varepsilon^{5/4}.
$$
\end{theorem}
\begin{proof}
Let $E^{ns} = u_{ns}-u^{a}_{ns}$. The above discussion implies that $E^{ns}$ satisfies the IBVP
$$
\left\{\begin{array}{l}
\partial_tE^{ns} - \varepsilon \partial_{xx} E^{ns} = \varepsilon R^{ns}, \qquad x>0,\\[3mm]
E^{ns}(t,0) - \varepsilon K_M\partial_xE^{ns}(t,0) = \varepsilon^{3/2} b_{ns}(t),\\[3mm]
E^{ns}(0,x) = 0. 
\end{array}\right.
$$
Multiplying $E^{ns}$ on the left of the equation yields
\begin{align*}
\frac{1}{2}\frac{d|E^{ns}|^2}{dt} - \varepsilon (E^{ns}E^{ns}_x)_x + \varepsilon |E^{ns}_x|^2 = \varepsilon E^{ns}R^{ns}.
\end{align*}
Integrating over $x\in[0,\infty)$ and using the BC, we have
\begin{align}
&\frac{1}{2}\frac{d}{dt}\|E^{ns}\|_{L^2(\mathbb{R}^+)}^2 \nonumber \\[2mm]
=& - \varepsilon E^{ns}(t,0) E^{ns}_x(t,0) - \varepsilon \int_{\mathbb{R}^+} |E^{ns}_x|^2 dx + \varepsilon \int_{\mathbb{R}^+} E^{ns}R^{ns} dx \nonumber \\[2mm]
=& - \frac{1}{K_M} E^{ns}(t,0) \Big[E^{ns}(t,0) - \varepsilon^{3/2}b_{ns}(t)\Big] - \varepsilon \int_{\mathbb{R}^+} |E^{ns}_x|^2 dx + \varepsilon \int_{\mathbb{R}^+} E^{ns} R^{ns} dx \nonumber \\[2mm]
\leq & - \frac{1}{2K_M} |E^{ns}(t,0)|^2 + \frac{\varepsilon^{3}}{2K_M}b_{ns}^2(t) - \varepsilon \int_{\mathbb{R}^+} |E^{ns}_x|^2 dx +  \frac{1}{2} \int_{\mathbb{R}^+} |E^{ns}|^2 dx + \frac{\varepsilon^2}{2}\int_{\mathbb{R}^+} |R^{ns}|^2 dx.\label{eq:nsinequality}
\end{align}
Since $K_M>0$, it follows that
\begin{align*}
\frac{d}{dt}\|E^{ns}\|_{L^2(\mathbb{R}^+)}^2  
\leq &  ~\frac{\varepsilon^{3}}{K_M}b_{ns}^2(t) + \int_{\mathbb{R}^+} |E^{ns}|^2 dx + \varepsilon^2\int_{\mathbb{R}^+} |R^{ns}|^2 dx\leq C \varepsilon^{5/2} + \|E^{ns}\|_{L^2(\mathbb{R}^+)}^2 .
\end{align*}
By Gronwall's inequality, we get the estimate 
$$
\|E^{ns}(t,\cdot)\|_{L^2(\mathbb{R}^+)}\leq C(T)\varepsilon^{5/4}.
$$ 
At last, we integrate \eqref{eq:nsinequality} over $t\in[0,T]$ to obtain the estimate
$$
\int_0^T|E^{ns}(t,0)|^2dt \leq C(T)\varepsilon^{5/2}.
$$
This completes the proof.
\end{proof}

%% file: L_MoM.tex
\subsection{Moment equations}
Now we consider the 3D time-dependent linear moment system
\begin{eqnarray*}
	\pd{W_{\varepsilon}}{t} + \sum_{d=1}^3 A_d\pd{W_{\varepsilon}}{x_d} &=&
	-\frac{1}{\varepsilon}QW_{\varepsilon},
\end{eqnarray*}
with the BCs at $\{x_2=0\}$,
\begin{equation}\label{eq:Grad_bc}
	B(W_{\varepsilon}-b)=0,
\end{equation}
where $b\in\bbR^N$ is given by \eqref{eq:Gbc_0} and
$B\in\bbR^{n\times N}$.

We assume that the initial value is compatible with the boundary
condition and there is no initial layer. Then we can
focus on the boundary-layer behavior of the moment system. In
the Grad BCs \eqref{eq:Gbc_0}, the matrix
\[B=E[\hat{\chi}S,M_o],\]
which has rare theoretical properties as mentioned in Section
\ref{sec:Back}. Because $S$ is symmetric positive definite,
we can choose the test polynomials in Grad's framework as
$\xi_2\phi_{\+\alpha}$ with $\+\alpha\in\mathbb{I}_e$ and
combine the obtained BCs linearly to get the modified BCs
\begin{equation}\label{eq:GBC}
[\hat{\chi}M_o^T,M_o^T S^{-1}M_o](W_{\varepsilon}-b)=0,
\end{equation}
which lead to
\[B=[\hat{\chi}M_o^T,M_o^T S^{-1}M_o].\]
We can check that the modified BCs with the no mass flow 
condition are maximal positive, which is important for the
well-posedness of linear hyperbolic equations.

\begin{theorem}
	 The modified BCs \eqref{eq:GBC} with $\+u\cdot\+n=-u_2=0$
	 are maximal positive. Namely,
	 \begin{itemize}
		\item[(i)] If $\+u^w=0$ and $\theta^w=0$, the BCs can 
	 determine a linear space $\mathcal{N}$ with  
	 \[\mathrm{dim}\ \mathcal{N} = m,\]
	 where $m$ is the number of nonnegative eigenvalues of 
	 $-A_2$ counting the multiplicity. 
		\item[(ii)] For any $v\in\mathcal{N}$, 
	 \[-v^TA_2v\geq 0.\]
	 \end{itemize}
\end{theorem}
\begin{proof}
	 By definition, we have
	 \[A_2=\begin{bmatrix} 0 & M_o \\ M_o^T & 0 \end{bmatrix}.\]
	 The no mass flow condition says that if $v\in\mathcal{N}$,
	 then 
	 \[v[\mathcal{N}(\+e_2)]=u_2=0.\]
	
	 Note that $M_o^TS^{-1}M_o>0$. We divide the vector into 
	 two parts, as
	 \[v = \begin{bmatrix} v_e \\ v_o \end{bmatrix},
		 \ v_e\in\bbR^m,\ v_o\in\bbR^n.\]
	 If $\hat{\chi}=0$, we have
	 \[v\in\mathcal{N}\quad\Leftrightarrow\quad v_o=0.\]
	 So $\mathrm{dim}\ \mathcal{N}=m$ and $v^TA_2v=0.$

	 If $\hat{\chi}>0$, since $M_o^T$ is upper triangular by
	 definition, $\rho^w$
	 will only occur in the first line of \eqref{eq:GBC} and we have
	 \[a_{11}(\rho-\rho^w)+a^Tv_e^* = v_o[\mathcal{N}(\+e_2)]=0,\]
	 where we write 
	 \[v_e = \begin{bmatrix} \rho \\ v_e^* \end{bmatrix},\quad
		 (M_o^TS^{-1}M_o)^{-1}M_o^T = \begin{bmatrix} a_{11} & a^T \\
		 \star & \star \end{bmatrix}.\]
	 Since $M_o^TS^{-1}M_o>0$ and $M_o$ is lower triangular with
	 full column rank,
	 $a_{11}$ equals the product of the first row and first column
	 of $(M_o^TS^{-1}M_o)^{-1}$ and $M_o^T$, which is not zero.
	 So for any $v_e\in\bbR^m$, the condition \eqref{eq:GBC} with
	 $u_2=0$ can determine the unique $v_o$ and $\rho^w$, which means
	 that 
	 \[\mathrm{dim}\ \mathcal{N}=m.\]
	 Next, because of the same reason, we can write
	 \[ M_o(M_o^TS^{-1}M_o)^{-1}M_o^T = \begin{bmatrix} c_0 & c_1^T\\
	 c_1 & S^*\end{bmatrix}\geq 0,\]
	 where $c_0>0$ and $S^*\geq 0$. Now $\rho^w$ should also
	 satisfy
	 \[c_0(\rho-\rho^w)+c_1^Tv_e^*=0,\]
	 which means that 
	 \[M_o(M_o^TS^{-1}M_o)^{-1}M_o^T(v_e-b_e) =  
	 \begin{bmatrix} 0 & 0\\
	 0 & S^*-c_0^{-1}c_1c_1^T\end{bmatrix}v_e:= \bar{S}v_e.
	 \]
	 Here $\bar{S}$ is symmetric positive semi-definite,
	 too. Thus, for any $v\in\mathcal{N},$ we have
	 \begin{eqnarray*}
		 -\frac{1}{2}v^TA_2v &=& -v_e^TM_ov_o \\
		 &=& \hat{\chi}v_e^T M_o(M_o^TS^{-1}M_o)^{-1}M_o^T
		 (v_e-b_e) \\
		 &=& \hat{\chi}v_e^T \bar{S} v_e \geq 0.
	 \end{eqnarray*}
	 This completes the proof.
 \end{proof}


%% file: L_AA.tex
\subsection{Equations for coefficients}

Motivated by the Couette flow case, we also consider the three scales ansatz:
\begin{equation}\label{eq:Section4ansatz}
	W_\varepsilon\left(t,\+x^w;x_2\right)
=\os{W}\left(t,\+x^w;x_2\right)
+\vs{W}\left(t,\+x^w;y\right)
+\ks{W}\left(t,\+x^w;z\right).
\end{equation}
Here the variables are 
$$
y = \frac{x_2}{\sqrt{\varepsilon}},\qquad z = \frac{x_2}{\varepsilon},\qquad \+x^w=(x_1,x_3),
$$
$\os{W}$ is the outer solution while $\vs{W}$ and
$\ks{W}$ are boundary-layer corrections satisfying 
$$
\vs{W}(t,\+x^w;\infty)=\ks{W}(t,\+x^w;\infty)=0.
$$
These can be expanded as 
\[ 
\os{W} = \sum_{j=0}^{\infty}\ak^j\os{W}{j},\qquad 
\vs{W} = \sum_{j=0}^{\infty}\ak^j\vs{W}{j},\qquad
\ks{W} = \sum_{j=0}^{\infty}\ak^j\ks{W}{j}.
\]
Next we will derive the equations for each coefficient $\os{W}{j}$, $\vs{W}{j}$ and $\ks{W}{j}$.

\subsubsection{Outer solution}
\label{sec:401}
At the scale of bulk flow, we know that the outer solution $\os{W}$ should satisfy the linear moment equations. 
Plugging the expansion of $\os{W}$ into \eqref{eq:Grad} and matching the order of $\varepsilon$, we have
\begin{eqnarray}
\label{eq:os}
\left\{\begin{aligned}
	&O(\varepsilon^{-1}):\quad ~~ 0 = -Q\os{W}{0}, \\[2mm]
	&O(\varepsilon^{-1/2}):\quad  0 = -Q\os{W}{1}, \\[2mm]
	&O(\varepsilon^{j/2}):\quad ~ 
	 \pd{\os{W}{j}}{t}+\sum_{d=1}^3 A_d\pd{\os{W}{j}}{x_d} 
= -Q\os{W}{j+2},\quad j\geq 0.
\end{aligned}\right.
\end{eqnarray}

We can analyze these linear equations by the null space method,
which will divide the equations into two parts according to the
orthogonal projection onto the null space of $Q$. In the language
of matrices, we assume $G\in\bbR^{N\times p}$ is an orthonormal 
basis matrix of the null space of $Q$, and $H\in\bbR^{N\times(N-p)}$
is an orthogonal complement of $G$. Then we have
\[I_N=GG^T+HH^T,\]
where $I_N\in\bbR^{N\times N}$ is the identity matrix. 
One can multiply
\eqref{eq:os} from the left by $G^T$ to obtain 
\begin{equation}
	\pd{G^T\os{W}{j}}{t}+\sum_{d=1}^3 G^TA_d\pd{\os{W}{j}}{x_d} 
	= 0,\quad j\geq 0,
\end{equation}
which are called the equilibrium equations. Multiplying 
\eqref{eq:os} from the left by $H^T$ and noting that
\[H^TQ=H^TQ(GG^T+HH^T)=H^TQHH^T\]
with $H^TQH>0$, we have 
\begin{eqnarray}
\label{eq:osnc}
\left\{
\begin{aligned}
&O(\varepsilon^{-1}):\quad ~~ 0 =  H^T\os{W}{0},\\[2mm]
&O(\varepsilon^{-1/2}):\quad  0 =  H^T\os{W}{1},\\[2mm]
&O(\varepsilon^{j/2}):\quad ~~
	\pd{H^T\os{W}{j}}{t}+\sum_{d=1}^3 H^TA_d\pd{\os{W}{j}}{x_d} 
= -(H^TQH)H^T\os{W}{j+2},\ j\geq 0,
\end{aligned}\right.
\end{eqnarray}
which are known as the constitutive relations. 

Following this way, we can successively get the equilibrium equations
about $G^T\os{W}{j}$, where $H^T\os{W}{j}$ is seen as given from the
algebraic relations \eqref{eq:osnc}. We put the details of the routine
computation in Appendix \ref{app:A}. In conclusion, the resulting 
equilibrium equations are linearized Euler-type equations
\begin{eqnarray}\label{eq:ols}
\begin{aligned}
	\pd{\os{\rho}{j}}{t} + \sum_d \pd{\os{u}{j}{d}}{x_d} &= 0,\\
\pd{\os{u}{j}{i}}{t} + \pd{\left(\os{\rho}{j}+\os{\theta}{j}
	\right)}{x_i} &= -\sum_d \pd{\os{\sigma}{j}{id}}{x_d},\quad i=1,2,3,\\
\frac{3}{2}\pd{\os{\theta}{j}}{t} + \sum_d\pd{\os{u}{j}{d}}{x_d}
	&= -\sum_d\pd{\os{q}{j}{d}}{x_d},
\end{aligned}
\end{eqnarray}
where $\os{\sigma}{j}{id}$ and $\os{q}{j}{d}$ 
are derived from \eqref{eq:osnc}. When $j=0$ or $j=1$, we have
\[\os{\sigma}{j}{id}=0,\quad \os{q}{j}{d}=0.\]
When $j=2$, we have
\begin{eqnarray}\label{eq:cc01}
	\os{\sigma}{j}{id} &=& -\gamma_1\left(\pd{\os{u}{j-2}{i}}{x_d}
	   +\pd{\os{u}{j-2}{d}}{x_i}-\frac{2}{3}\delta_{id}
	   \nabla\cdot\os{\+u}{j-2}\right),\quad i,d=1,2,3. \\
	\os{q}{j}{d} &=& -\frac{5}{2}\gamma_2\pd{\os{\theta}{j-2}}{x_d},
		\label{eq:cc02}
\end{eqnarray}
where the constants $\gamma_1$ and $\gamma_2$ are defined in
\eqref{eq:cons0}\eqref{eq:cons1}. In the BGK model, we have
$\gamma_1=\gamma_2=1.$ For the hard sphere gas, when the moment
order $M$ is large, we approximatively \cite{Rein1990} have
$\gamma_1=1.270042$ and $\gamma_2=1.922284$, which are close 
to the values calculated by the Chapman-Enskog expansion of
the Boltzmann equation.

According to the classical theory of the linear hyperbolic 
system \cite{Benzoni,Hil2013}, the linearized Euler-type equations
\eqref{eq:ols} need and only need one boundary condition at
the wall, i.e., the value of $\os{u}{j}{2}$ should be prescribed
at $x_2=0$. 

\subsubsection{Viscous layer solution}
We assume there exists a viscous layer solution $\vs{W}$ which changes
dramatically at the normal direction to the boundary and vanishes
outside the viscous layer with thickness of $O\ak$. 
To derive the equations for $\vs{W}^{(j)}$, we collect the outer solution and the viscous layer correction together:
\begin{gather}
 \es{W} = \sum_{j=0}^{\infty}\ak^j\es{W}{j},\qquad	\es{W}{j}(t,\+x) = \os{W}{j}(t,\+x) + \vs{W}{j}(t,\+x^w;y),\label{eq:vinf0} 
\end{gather}
with $y = x_2/\sqrt{\varepsilon}\geq 0$ and $\+x^w=(x_1,x_3)$. The outer solution $\os{W}$ and $\es{W}$
should both satisfy the moment system \eqref{eq:Grad}. Since the
system is linear, we find that $\vs{W}=\es{W}-\os{W}$ also satisfies \eqref{eq:Grad}. Then plugging $\vs{W}$ into the equation and 
comparing each order of $\varepsilon$, we have 
\begin{eqnarray}\label{eq:vs}
\left\{\begin{aligned}
		&O(\varepsilon^{-1}):\quad ~~ 0 = -Q\vs{W}{0},\\[2mm]
		&O(\varepsilon^{-1/2}):\quad 
		A_2\pd{\vs{W}{0}}{y} = -Q\vs{W}{1}, \\[2mm]
		&O(\varepsilon^{j/2}):\quad ~~ \pd{\vs{W}{j}}{t} +\sum_{d\neq
		2}A_d\pd{\vs{W}{j}}{x_d}
		+ A_2\pd{\vs{W}{j+1}}{y} = -Q\vs{W}{j+2},\ j\geq 0.
\end{aligned}\right.
\end{eqnarray}
Utilizing an analogous null space method, the equilibrium 
variables $G^T\vs{W}{j}$ satisfy two degenerate algebraic 
relations and three linear parabolic equations, which is similar to
 the linearized Prandtl boundary layer equations \cite{Sone2007}.
Multiplying \eqref{eq:vs} from the left by $H^T$, we obtain the constitutive 
relations about $H^T\vs{W}{j}$ (details can be found in 
Appendix \ref{app:A}).

When $j=0$, the viscous layer solutions satisfy
the two algebraic relations
\begin{gather}
	\vs{u}{0}{2} = 0,\quad \vs{p}{0} = 0
\end{gather}
with
\[\vs{p}{j}:= \vs{\rho}{j}+\vs{\theta}{j}\]
and the following three parabolic equations
\begin{subequations}
\begin{eqnarray}
	\label{eq:vs_01}
	\pd{\vs{u}{0}{i}}{t} &=& \gamma_1\pd{^2\vs{u}{0}{i}}{y^2}
	,\quad i\neq 2,\\
	\pd{\vs{\theta}{0}}{t} &=& \gamma_2\pd{^2\vs{\theta}{0}}{y^2},
	\label{eq:vs_02}
\end{eqnarray}
\end{subequations}
where $\gamma_1$ and $\gamma_2$ are constants defined in 
\eqref{eq:cons0}\eqref{eq:cons1}.

When $j=1$, we have the equations
\begin{subequations}
	\begin{eqnarray}
		\vs{u}{1}{2} &=& \int_y^{\infty}\!\!\left(
		\pd{\vs{\rho}{0}}{t}+\sum_{d\neq 2}\pd{\vs{u}{0}{d}}
		{x_d}\right)(t,\+x^w;s)\,\mathrm{d}s,\\
		\vs{p}{1} &=& 0, \\[2mm]
		\pd{\vs{u}{1}{i}}{t} &=& \gamma_1 \pd{^2\vs{u}{1}{i}}{y^2}
		,\quad i\neq 2,\\
		\pd{\vs{\theta}{1}}{t} &=& \gamma_2 
		\pd{^2\vs{\theta}{1}}{y^2},
	\end{eqnarray}
\end{subequations}
which contain two algebraic relations determining $\vs{u}{1}{2},\
\vs{p}{1}$, and three parabolic equations.

For the higher order constitutive relations, we calculate that 
	\begin{subequations}
		\begin{eqnarray*}
			\vs{\sigma}{2}{id} &=& -\gamma_1\left(\pd{\vs{u}{0}{d}}{x_i}
			+\pd{\vs{u}{0}{i}}{x_d}-\frac{2}{3}\delta_{id}
			\left(\pd{\vs{u}{0}{1}}{x_1}+\pd{\vs{u}{0}{3}}{x_3}+
			\pd{\vs{u}{1}{2}}{y}\right)\right)
		 - \frac{2}{3}\gamma_3\delta_{id}
		\pd{^2\vs{\theta}{0}}{y^2},\ i,d\neq 2, \\	
			\vs{\sigma}{2}{22} &=& -\gamma_1\left(2\pd{\vs{u}{1}{2}}{y}
		-\frac{2}{3}
			\left(\pd{\vs{u}{0}{1}}{x_1}+\pd{\vs{u}{0}{3}}{x_3}+
			\pd{\vs{u}{1}{2}}{y}\right)\right) 
		+ \frac{4}{3}\gamma_3\pd{^2\vs{\theta}{0}}{y^2},\\
			\vs{\sigma}{2}{2d} &=& -\gamma_1\pd{\vs{u}{1}{d}}{y},
	\ d\neq 2.\\
			\vs{q}{2}{i} &=& -\frac{5}{2}\gamma_2\pd{\vs{\theta}{0}}{x_i}
			+\gamma_3\pd{^2\vs{u}{0}{i}}{y^2},\quad i\neq 2,
			\\ \vs{q}{2}_{2} &=&
	-\frac{5}{2}\gamma_2\pd{\vs{\theta}{1}}{y},
		\end{eqnarray*}
	\end{subequations}
where the constant $\gamma_3$ is given by \eqref{eq:Agam3} and
equals one in the BGK case. For the hard sphere gas, the value
of $\gamma_3$ calculated by the LBE
\cite{Sone2007} is $1.947906$.

According to the theory of parabolic equations in half-space,		  
every order of the viscous layer solutions exactly need 
three boundary conditions at $y=0$, e.g., the values of
$\vs{u}{j}{d}(0),\ d\neq 2,$ and $\vs{\theta}{j}(0)$.  

\begin{remark}
As shown in the Couette flow, the viscous layer would vanish
for the linear steady problem, but can not be ignored for the 
instantaneous flow. In the latter case, the outer solution
can't satisfy the boundary condition of the moment equations
up to $O(1)$. So we should insert a viscous layer to match the
boundary condition. We note that for the non-linear problem,
the viscous layer also appears in the steady flow \cite{Sone2007}.
\end{remark}

\subsubsection{Knudsen layer solution}
Since the moment system need $n$ boundary conditions at the wall
(cf. Section \ref{sec:Back}), generally we should insert a Knudsen
layer $\ks{W}$ to match the BCs up to the higher order. 
We assume there exists a Knudsen layer solution which changes
dramatically at the normal direction to the boundary and vanishes
outside the Knudsen layer with thickness of $O(\varepsilon)$. 

Clearly, the asymptotic solution $\es{W}_{\varepsilon}$ defined in \eqref{eq:Section4ansatz} should satisfy the the moment system \eqref{eq:Grad}. Besides, from the previous discussion, we know that $\es{W}=\os{W}+\vs{W}$ should also satisfy the moment system. Since the equations are 
linear, we find that the Knudsen layer solution $\ks{W}=\es{W}_{\varepsilon}
-\es{W}$ satisfies \eqref{eq:Grad}. 
Substituting the ansatz 
\[\ks{W}\left(t,\+x^w;z\right) = \sum_{j=0}^{\infty}\ak^j\ks{W}{j}\left(t,\+x^w;z\right)
\]
into the moment system and matching the order of $\varepsilon$, we have
\begin{eqnarray}
\label{eq:kseq}
\begin{aligned}
	A_2\pd{\ks{W}{j}}{z} &= -Q\ks{W}{j},\ j=0,1,\\[2mm]
	A_2\pd{\ks{W}{j}}{z} &= -Q\ks{W}{j}-\pd{\ks{W}{j-2}}{t}
	-\sum_{d\neq 2}\pd{\ks{W}{j-2}}{x_d},\ j\geq 2. 
\end{aligned}
\end{eqnarray}
The Knudsen layer solution satisfies a system of linear
homogeneous (or non-homogeneous) ODEs in half-space.
According to the result by Bobylev and Bernhoff \cite{Bern2008}
(also see Appendix \ref{app:A}),
the Knudsen layer solution need and only need $n-4$ boundary
conditions, where $n$ is exactly the number of boundary conditions
required by the linear moment equations, and $4$ coincides with
the number required by the outer solution and viscous layer solution.

\subsection{BCs for coefficients} 
Now we plug the above expansion into the BCs \eqref{eq:Grad_bc} for
the moment equations. According to solvability conditions of the 
Knudsen layer solution, we can obtain slip boundary conditions for
the outer solution and viscous layer solution. The Robin coefficients
in the BCs are determined by some elemental half-space problems.
Following this way, we may successively solve all the asymptotic 
equations.

As a generalization of the Couette flow, we state the solvability
theorem as follows, which is a main result of our earlier paper
\cite{Yang2022a}:
\begin{theorem}\label{thm:41}
	Consider the half-space problem for $\ks{W}=\ks{W}(z)$:
\begin{eqnarray}
\label{eq:KL1}
\begin{aligned}
	&A_2\pd{\ks{W}}{z} = -Q\ks{W},\quad z\geq 0,\\[2mm]
	&B(\ks{W}(0)-h) = 0,\quad \ks{W}(\infty)=0.
\end{aligned}
\end{eqnarray}
	Here $A_2$ and $Q$ are given by \eqref{eq:defofAQ}, 
	$B$ given by \eqref{eq:GBC}, and $h\in\bbR^N$.
	We can decompose any $h$ as
	\[h= (GG^T+HH^T)h = G_e(G_e^Th) + H(H^Th),\]
	where
	$G_e=[\varphi_0,\varphi_1,\varphi_3,\varphi_4]\in\bbR^{N\times 4}$
	defined in Appendix \ref{app:A}.
	Then, for any given $H^Th$,
	the system \eqref{eq:KL1} is uniquely solvable which gives
	the value of 
	$\ks{W}$ and $G_e^Th$.
\end{theorem}

	When $j=0$, the Knudsen layer solution satisfies 
\begin{eqnarray}
\label{eq:KL0}
\begin{aligned}
	A_2\pd{\ks{W}{0}}{z} &= -Q\ks{W}{0},\\[2mm]
	B(\es{W}{0}+\ks{W}{0}-\es{b}{0}) &= 0,\quad \text{at}\ z=0,
\end{aligned}
\end{eqnarray}
where we expand $b$ analogously with 
\[
	\es{b}{0}[\mathcal{N}(\+0)]=\rho^{w,(0)},\
 	\es{b}{0}[\mathcal{N}(\+e_i)]=u_i^{w,(0)},\
	\es{b}{0}[\mathcal{N}(2\+e_i)]=\theta^{w,(0)}/\sqrt{2}.
\]
Here we can calculate that 
	\[G_e^T (\es{W}{j}-\es{b}{j}) = \left[\es{\rho}{j}-\rho^{w,(j)},
	\es{u}{j}{1}-u^{w,(j)}_1,
	\es{u}{j}{3}-u^{w,(j)}_3, \frac{\sqrt{6}}{2}\left(
			\es{\theta}{j}-\theta^{w,(j)}\right)\right]^T.\]
By Theorem \ref{thm:41}, there is only the 
zero solution $\ks{W}{0}=0$ and $G_e^T (\es{W}{0}-\es{b}{0})=0$ 
since $H^T(\es{W}{0}-\es{b}{0})=0.$ This shows that there is
no Knudsen layer when $j=0$ and we have the no-slip BCs
\begin{subequations}
	\begin{eqnarray}
		\label{eq:vkbc0_0}
		\es{u}{0}{i} &=& u_{i}^{w,(0)},\quad i=1,2,3, \\
		\es{\theta}{0} &=& \theta^{w,(0)}.
		\label{eq:vkbc0_2}
	\end{eqnarray}
\end{subequations}

Analogously, when $j=1$, we regard $H^T\es{W}{1}$ as the driven
term and $G_e^T(\es{W}{1}-\es{b}{1})$ can be solved from
\eqref{eq:KL1}. From Appendix \ref{app:A}, we have 
$H^T\os{W}{1}=0$ and $H^T\vs{W}{1}$ represented by derivatives
of $G^T\vs{W}{0}$. Thus, we have the following slip BCs for 
the outer solution and viscous layer solution:
\begin{subequations}
	\begin{eqnarray} \label{eq:vkbc1_0}
		\es{u}{1}{2} &=& 0,\\
		\es{u}{1}{i}-u_1^{w,(1)} 
		&=& \sqrt{2}k_0\pd{\vs{u}{0}{i}}{y},\quad i=1,3,\\
		\es{\theta}{1}-\theta^{w,(1)} &=&
		\sqrt{2}t_1\pd{\vs{\theta}{0}}{y},
		\label{eq:vkbc1_2}
	\end{eqnarray}
\end{subequations}
where $k_0$ and $t_1$ are constants solved from the elemental
problems defined below.

When $j=2$, since $\ks{W}{0}=0$, the Knudsen layer solution also
satisfies the linear homogeneous equations in half-space. From
Appendix \ref{app:A}, now $H^T\es{W}{2}$ can be represented by
derivatives of $G^T\es{W}{0},\ G^T\es{W}{1}$ and $H^T\es{W}{1}.$ 
According to the solvability condition, after a careful calculation,
we have the slip BCs
\begin{subequations}
	\begin{eqnarray}\label{eq:vkbc2_0}
		\es{u}{2}{2} &=& 0,\\
		\es{u}{2}{i}-u_i^{w,(2)} &=& \sqrt{2}k_0
		\left(\pd{\os{u}{0}{i}}{x_2}
		+\pd{\os{u}{0}{2}}{x_i}+\pd{\vs{u}{1}{i}}{y}\right)
		\notag \\ && + 2t_0\pd{\os{\theta}{0}}{x_i}
		+ 2k_2\pd{^2\vs{u}{0}{i}}{y^2},
		\quad i=1,3,\\
		\es{\theta}{2}-\theta^{w,(2)}  
		&=& \sqrt{2}t_1\left(\pd{\vs{\theta}{1}}{y}
		+\pd{\os{\theta}{0}}{x_2}\right) + 
		2t_2 \pd{^2\vs{\theta}{0}}{y^2}
		\notag \\ &&
		+ k_1 \left(\pd{\vs{u}{1}{2}}{y}+
		\pd{\os{u}{0}{2}}{x_2}\right),
		\label{eq:vkbc2_2}
	\end{eqnarray}
\end{subequations}
where the constants $k_i$ and $t_i$ are solved from the following
elemental problems.

The elemental problems arise from the linear superposition principle.
For example, $H^T\vs{W}{1}$ is a linear combination of derivatives
of $G^T\vs{W}{0}$. So we can decompose the vector $H^T\vs{W}{1}$ into
several parts, which have the different ``driven'' terms such as
the gradient of velocity and the gradient of temperature.
In essence, we replace $h$ in \eqref{eq:KL1} by different driven
terms to obtain the elemental problems. Here we have six 
elemental problems in the form
\begin{eqnarray}
\label{eq:KLA1}
\begin{aligned}
	A_2\pd{\ks{W}}{z} = -Q\ks{W}&,\ \ks{W}=\ks{W}(z),\ z\geq 0,
	\ \ks{W}(\infty)=0,\\
	B\left(G_eG_e^Th+\ks{W}\right) &= 
	BH(H^TQH)^{-1}\mathcal{A},\quad \text{at}\ z=0.
\end{aligned}
\end{eqnarray}
More precisely, using the vectors $r_{id}$ and $s_d$ defined 
in \eqref{eq:cons0}\eqref{eq:cons1}, we have
\begin{itemize}
	\item Velocity slip problem. 
		\[\mathcal{A}=r_{12},\quad k_0=\frac{\sqrt{2}}{2}\varphi_1^Th.\]
	Note that here $\mathcal{A}$ is given, and $k_0$ is part of
		$G_e^Th$, solved from the system.

	\item Temperature jump problem. 
		\[\mathcal{A}=s_2,\quad t_1=
		\frac{\sqrt{3}}{3}\varphi_4^Th.\]
	
	\item Thermal creep problem. 
		\[\mathcal{A}=s_1,\quad t_0 = \frac{1}{2}\varphi_1^Th.\]

	\item The fourth problem. 
		\[\mathcal{A}=\sqrt{2}H^Tr_{22},\quad 
		k_1=\frac{\sqrt{6}}{3}\varphi_4^Th.\]

	\item Second order viscous slip problem. 
		\[\mathcal{A}= -H^TA_2H(H^TQH)^{-1}r_{12},\quad
		k_2 = \frac{1}{2}\varphi_1^Th.
		\]

	\item Second order temperature jump problem. 
		\[\mathcal{A}= -H^TA_2H(H^TQH)^{-1}s_{2},\quad
		t_2 = \frac{\sqrt{6}}{6}\varphi_4^Th.
		\]
\end{itemize}

Till now, we have derived the asymptotic equations and their slip
boundary conditions from the linear moment system by Hilbert expansion.
The procedure to determine the asymptotic solutions are concluded
as follows:
\begin{enumerate}
	\item Solve $\vs{u}{j}{2}$ from the algebraic relation of the
   viscous layer solution.
	\item Determine $\os{u}{j}{2}=-\vs{u}{j}{2}$ at the boundary
 	and solve the linearized Euler-type equations \eqref{eq:ols} to
	get $\os{W}{j}$.
	\item Determine $\vs{u}{j}{i},\ i\neq 2,$ and $\vs{\theta}{j}$
	from the slip boundary conditions. Then solve the parabolic
	equations to get $\vs{W}{j}$.
	\item Solve the Knudsen layer solution $\ks{W}{j}$ from the
	half-space problem.
	\item Let $j=j+1$ and return to the first step. 
\end{enumerate}

%% file: L_AABC.tex
\subsection{Construction}
The NS equations never directly appear in the Hilbert
expansion \cite{CC1989}. However, we may formally retrieve
the NS equations and construct their BCs by collecting 
the coefficients of $\os{W}{j}$ and $\vs{W}{j}$ (outer solutions and the viscous layer solutions) based on the matching requirement.

Inspired by the Couette flow case, we consider 
\[ W_{m} = \sum_{j=0}^2\ak^j\os{W}{j}+\sum_{j=0}^2\ak^j\vs{W}{j}.\]
Then we can write equations about $W_m$ from 
\eqref{eq:os} and \eqref{eq:vs}. For example, we have
\begin{eqnarray*}
	\pd{\rho_m}{t} + \sum_d\pd{u_{d,m}}{x_d} &=&
	0 + \pd{}{t}\sum_{j=0}^2\ak^j\vs{\rho}{j} + \sum_d\pd{}{x_d}
	\sum_{j=0}^2\ak^j\vs{u}{j}{d} \\
	&=& \ak^2\pd{\vs{\rho}{2}}{t} + \ak^2\sum_{d\neq 2}
	\pd{\vs{u}{2}{d}}{x_d} \\
	&\triangleq& \varepsilon R_0(t,\+x^w;y),
\end{eqnarray*}
where $R_0$ is of unity order and vanishes
when $y\rightarrow +\infty.$ Treating the other equilibrium equations
in the same way, we have 
\begin{subequations}
\begin{eqnarray}\label{eq:LM0}
	\pd{{\rho}_{m}}{t} + \sum_d \pd{{u}_{d,m}{}}{x_d} &=& 0 +
	\varepsilon R_0(t,\+x^w;y),\\
	\pd{{u}_{i,m}{}}{t} + \pd{({\rho}_{m}+{\theta}_{m})}{x_i}
	&=& \varepsilon\sum_d \pd{}{x_d}\left(\gamma_1
	\left(\pd{{u}_{i,m}{}}{x_d}  \label{eq:LM1}
	   +\pd{{u}_{d,m}{}}{x_i}-\frac{2}{3}\delta_{id}
	   \nabla\cdot{\+u}_{m}\right)\right) \\ \notag
	   && +\varepsilon R_i(t,\+x^w;y),
	   \quad i=1,2,3,\\ \label{eq:LM2}
	\frac{3}{2}\pd{{\theta}_{m}}{t} + \sum_d \pd{{u}_{d,m}{}}{x_d}
	&=& \frac{5}{2}\varepsilon \sum_d \pd{}{x_d}\left(\gamma_2 
	\pd{{\theta}_{m}}{x_d}\right)+\varepsilon R_4(t,\+x^w;y), 
\end{eqnarray}
\end{subequations}
where $R_i$ is of unity order and vanishes when $y\rightarrow +\infty.$
At the same time, from the BCs
\eqref{eq:vkbc0_0}-\eqref{eq:vkbc0_2},
\eqref{eq:vkbc1_0}-\eqref{eq:vkbc1_2} and
\eqref{eq:vkbc2_0}-\eqref{eq:vkbc2_2}, the BCs for $W_m$ should be
\begin{subequations}
	\begin{eqnarray} \label{eq:nsb1}
		u_{2,m} &=& 0,\\ \label{eq:nsb2}
		u_{i,m} - u_{i}^w &=& {\sqrt{2}}k_0\varepsilon
		\left(\pd{u_{i,m}}{x_2}+ \pd{u_{2,m}}{x_i}\right) 
		+ {2}t_0\varepsilon\pd{\theta_m}{x_i} \notag \\
		&& + 2k_2\varepsilon^2\pd{^2 u_{i,m}}{x_2^2} +
		O(\varepsilon^{3/2}),\quad i=1,3,\\ \label{eq:nsb3}
	\theta_m - \theta^w  &=& {\sqrt{2}}
	t_1\varepsilon\pd{\theta_m}{x_2}
		+ 2t_2 \varepsilon^2 \pd{^2{\theta_m}}{x_2^2}
		+ k_1\varepsilon\pd{u_{2,m}}{x_2}+O(\varepsilon^{3/2}).
\end{eqnarray}
\end{subequations}

\noindent \textbf{NS equations.}
Discarding the residual terms $R_i$ in \eqref{eq:LM0}-\eqref{eq:LM2},
we formally obtain the (linearized dimensionless) Navier-Stokes 
equations 
\begin{subequations}
\begin{eqnarray}\label{eq:LNS0}
	\pd{{\rho}_{}}{t} + \sum_d \pd{{u}_{d}{}}{x_d} &=& 0,\\
	\pd{{u}_{i}{}}{t} + \pd{({\rho}_{}+{\theta}_{})}{x_i}
	&=& \varepsilon\sum_d \pd{}{x_d}\left(\gamma_1
	\left(\pd{{u}_{i}{}}{x_d}  \label{eq:LNS1}
	   +\pd{{u}_{d}{}}{x_i}-\frac{2}{3}\delta_{id}
	   \nabla\cdot{\+u}_{}\right)\right),
	   \quad i=1,2,3,\\ \label{eq:LNS2}
	\frac{3}{2}\pd{{\theta}_{}}{t} + \sum_d \pd{{u}_{d}{}}{x_d}
	&=& \frac{5}{2}\varepsilon \sum_d \pd{}{x_d}\left(\gamma_2 
	\pd{{\theta}{}}{x_d}\right). 
\end{eqnarray}
\end{subequations}
This form is consistent with the NS equations introduced 
in \eqref{linearNS} as $\mu=\gamma_1$ and $\lambda = \gamma_2$.
\\

\noindent\textbf{Slip BCs for the NS equations.} 
Moreover, discarding $O(\varepsilon^{3/2})$ terms in
\eqref{eq:nsb1}-\eqref{eq:nsb3}, we obtain the constructed BCs: 
\begin{subequations}
	\begin{eqnarray} \label{eq:Nsb1}
		u_{2} &=& 0,\\ \label{eq:Nsb2}
		u_{i} - u_{i}^w &=& {\sqrt{2}}k_0\varepsilon
		\left(\pd{u_{i}}{x_2}+ \pd{u_{2}}{x_i}\right) 
		+ {2}t_0\varepsilon\pd{\theta}{x_i} 
		+ 2k_2\varepsilon^2\pd{^2 u_{i}}{x_2^2},
		\quad i=1,3,\\ \label{eq:Nsb3}
	\theta - \theta^w  &=& {\sqrt{2}}
	t_1\varepsilon\pd{\theta}{x_2}
		+ 2t_2 \varepsilon^2 \pd{^2{\theta}}{x_2^2}
		+ k_1\varepsilon\pd{u_{2}}{x_2}.
\end{eqnarray}
\end{subequations}

\begin{remark}
Note that the second-order spatial derivatives at the normal direction appear in the above BCs. As to the constructed BCs in Section \ref{sec:3}, the term with second-order derivatives is transferred to the temporal derivative (see \eqref{eq:couette2edderi}). For the general system, it seems not straightforward to use the same technique. 
\end{remark}

Assume there is no initial layer and the initial values are same
for equations of $W_m$ and the NS equations. Subtracting the NS
equations \eqref{eq:LNS0}-\eqref{eq:LNS2} by
the truncated moment equations \eqref{eq:LM0}-\eqref{eq:LM2}, we
shall find that the error functions $\rho-\rho_m,\ u_i-u_{i,m}$ and
$\theta-\theta_m$ satisfy \eqref{eq:LM0}-\eqref{eq:LM2} with
the zero initial value and the BCs
\eqref{eq:nsb1}-\eqref{eq:nsb3} where $u_i^w=\theta^w=0$.
In case of the Couette flow, we have proved that 
\[\|u_1-u_{1,m}(t,\+x^w;\cdot)\|_{L^2(\bbR^+)}
\leq C(T)\varepsilon^{5/4} \]
for $t\in[0,T].$ For the general situation,
we have no rigorous proof but expect that
the $L^2$ errors of these five field variables would be
higher order terms about $\varepsilon$.

In conclusion, we have formally constructed the linearized 
NS equations with slip BCs from the linear moment equations.
The obtained BCs contain not only the first-order derivatives
but also the second-order derivatives at the normal direction
to the wall. Besides the formal derivation, our analysis
on the Couette flow exhibits that the second-order terms
are necessary for instantaneous flows to get a first-order
approximation solution. 

\subsection{Comparison and remarks}
\label{sec:6}
We will compare the obtained slip BCs with related work
from two aspects, i.e., the form and the values of slip
coefficients. 

Firstly, we compare the form of the obtained slip BCs.
Putting aside the second-order derivatives in the BCs,
the BCs \eqref{eq:Nsb1}-\eqref{eq:Nsb3} 
coincide with the linearized version of the classical
first-order slip BCs \cite{Aoki2017} for the NS equations.
As mentioned in Section \ref{sec:Back}, many efforts have
been paid to deriving the second-order slip BCs, both
from the physical and mathematical insights. Compared 
with the second-order slip BCs in \cite{Dei1964}, our
second-order terms only contain the normal derivatives,
not including the tangential derivatives or mixed derivatives.
The difference arises from the methodology behind the derivation.

In this paper, we aim to derive BCs for the NS equations
such that the solution of the obtained system is a
first-order, not second-order, approximation solution 
to the moment equations. Due to the existence of the viscous
layer, the second-order normal derivatives, e.g.,
$\varepsilon^2\pd{^2\theta}{x_2^2}$, are
in essence of the magnitude $O(\varepsilon)$, not $O(\varepsilon^2).$
While the second-order tangential or mixed derivatives are
of the magnitude $O(\varepsilon^2)$, ignored as the higher-order
quantities. In \cite{Dei1964}, the second-order derivatives arise
from the expansion to the second order $O(\varepsilon^2)$, without 
considering the half order $O\ak.$

Secondly, we compare the obtained slip coefficients. For this
purpose, we refer to Sone's generalized slip flow theory
\cite{Sone2007,2019Hat}. 
In the linear case \cite{2019Hat}, the generalized slip flow
theory applies Hilbert expansion to the Boltzmann equation,
obtaining the so-called fluid-dynamic equations and Knudsen
layer corrections. The expansion is on the parameter $\varepsilon$ 
and explicitly given up to the second order. Slip boundary
conditions and their slip coefficients are determined by
solving linear layer equations \cite{2019Hat}.  

However, these slip boundary conditions are proposed for their
fluid-dynamic equations, not the NS equations. Also, \cite{2019Hat}
does not consider the scale $O(\sqrt{\varepsilon})$. Despite
these differences, the elemental half-space problems in \cite{2019Hat} 
provide a good reference for the elemental problems in this
paper. Numerically, we can see that the solved constants
$k_0$, $t_0$, $t_1$, $k_1, k_2$ and $t_2$  
individually converge to the coefficients $b_1^{(1)}, b_2^{(1)},
c_1^{(0)}, c_5^{(0)}, b_4^{(1)}$ and $c_6^{(0)}$ in \cite{2019Hat}
when the moment order $M$ goes larger. This is because
we both use the Hilbert expansion and introduce the Knudsen
layer correction, while \cite{2019Hat} starts from the
Boltzmann equation and we start from the moment equations.
Incidentally, \cite{Sone2007,2019Hat} linearize the
Boltzmann equation around $\pi^{-3/2}\exp(-|\+\xi|^2)$
rather than $\mathcal{M}$ in this paper. So for convenience of
comparison, there are
some scaling constants in \eqref{eq:Nsb1}-\eqref{eq:Nsb3}
before the slip coefficients $t_i$ and $k_i$.

The numerical method to solve elemental problems is briefly 
stated in Appendix \ref{app:C}.
We compare the calculated coefficients with results in
\cite{2019Hat} for the BGK model when the accommodation
coefficient $\chi=1$. Table \ref{tab:01}
and Table \ref{tab:02} list the corresponding results. 
The parity of $M$ would affect the solution behavior for different
problems. To make the convergence trends clear, we put the results for $M$
with the same parity in one table. We can 
see that the relative error is less than $1\%$ when $M>10$.
So we may deduce that moment equations with mild moments can 
model the half-space problems well. 

\begin{table}[!htb] 
\centering 
\caption{The coefficients $k_0, t_0, k_2$ compared with \cite{2019Hat} for 
	the BGK model when $\chi=1$. 
		($k_0$ is compared with $b_1^{(1)}$, $t_0$
	with $b_2^{(1)}$, $k_2$ with $b_4^{(1)}$, where the latter
	coefficients appearing in \cite{2019Hat})}\label{tab:01}
\begin{tabular}{ccccccc} 
\toprule 
	 & \cite{2019Hat} & $M=4$ & $M=6$ & $M=8$ & $M=10$ & $M=12$ \\ 
\midrule 
$k_0$ & 1.01619 & 0.99247 & 1.00360 & 1.00772 & 1.00984 & 1.01112\\
$t_0$ & 0.38316 & 0.36988 & 0.37617 & 0.37848 & 0.37967 & 0.38039\\
$k_2$ & -0.76632 & -0.73976& -0.75233 & -0.75697 & -0.75934 &
-0.76077\\
\bottomrule 
\end{tabular}
\end{table}

\begin{table}[!htb] 
\centering 
\caption{The coefficients $k_1, t_1, t_2$ compared with \cite{2019Hat} for 
	the BGK model when $\chi=1$.
	   	($k_1$ is compared with $c_5^{(0)}$, $t_1$
	with $c_1^{(0)}$, $t_2$ with $c_6^{(0)}$, where the latter
	coefficients appearing in \cite{2019Hat})}\label{tab:02}
\begin{tabular}{ccccccc} 
\toprule 
	 & \cite{2019Hat} & $M=3$ & $M=5$ & $M=7$ & $M=9$ & $M=11$ \\ 
\midrule 
$k_1$ & 0.44046 & 0.42763 & 0.43922 & 0.44019 & 0.44040 & 0.44046\\
$t_1$ & 1.30272 & 1.12868 & 1.27183 & 1.28673 & 1.29213 & 1.29488\\
$t_2$ & -1.42758 &  & -1.38715 & -1.40694 & -1.41403 &
-1.41760\\
\bottomrule 
\end{tabular}
\end{table}

We make some remarks on the second-order slip coefficients $k_2$. 
In Kramers' problem where $\partial/\partial x_i=0,\ i\neq 2,$
our boundary condition is 
\[
	u_{1} - u_{1}^w = {\sqrt{2}}k_0\varepsilon
		\pd{u_{1}}{x_2}+ 2k_2\varepsilon^2\pd{^2 u_{1}}{x_2^2}.
\]
Cercignani \cite{CC1989} has solved the LBE
between two parallel walls to get
\[
	u_{1} - u_{1}^w = C_1\pd{u_{1}}{x_2}- C_2\pd{^2 u_{1}}{x_2^2},
\]
where $C_1=\displaystyle\frac{2}{\sqrt{\pi}}*1.016\approx 1.146,
\ C_2=\frac{1}{\pi}+\frac{C_1^2}{2}\approx 0.975.$ If we choose 
$\varepsilon=\displaystyle\sqrt{\frac{2}{\pi}}$ and $k_0,\ k_2$
as above, we have 
$\sqrt{2}k_0\varepsilon\approx 1.146$ and $2k_2\varepsilon^2\approx
-0.976$, which agrees with Cercignani's result. 

Thus, we can conclude that all terms 
in \eqref{eq:Nsb1}-\eqref{eq:Nsb3} have appeared in the literature.
But the methodology and starting point are new. Due to
the discrete essence of the moment method, we can 
write the analytical expressions of the second-order
slip coefficients for a moderate $M$.

\begin{remark}
	The nonlinear theory seems feasible. We
	can choose the nonlinear moment equations as the globally 
	hyperbolic moment equations \cite{CaiFanLi} and perform Hilbert
	expansion on them. 
	
	The computation is routine and 
	we have checked that the zeroth order outer equations are 
	nonlinear compressible Euler equations, while the next order
	equations are linearized Euler equations around the zeroth 
	order solutions. The zeroth order viscous layer equations 
	are nonlinear nonlocal Prandtl boundary layer equations, while
	the next order equations are some linear equations. 
	The considered Knudsen layer equations are also linear. We 
	can utilize the solvability condition of linear half-space 
	problems again to determine slip BCs for the
	equations of outer solutions and viscous layer solutions. 
	These results agree with the Hilbert expansion of the Boltzmann
	equation \cite{Sone2007}.

	However, it's not trivial to retrieve the NS equations in
	the nonlinear case. The theoretical analysis is also more
	difficult. 
\end{remark}

%% file: ANE.tex
In this section, we explore the velocity profile of the Couette
flow numerically. The gas is confined by two parallel plates as
in Section \ref{sec:3}. We assume the upper plate is at $x_2=1$
and the lower plate is at $x_2=0.$ At $t=0$, the gas is at rest.
Then the lower plate suddenly moves right with the velocity
$u^w_1(t,0)$. We assume the simulation time is relatively short 
and the gas at the upper plate is almost at rest. So we can 
impose the no-slip BCs at the upper plate and focus on the 
boundary-layer behavior around the lower plate.

The simplified NS equations for the Couette flow are 
\begin{gather*}
\pd{u_1}{t} = \varepsilon \gamma_1 \pd{^2u_1}{x_2^2},\\
	u_1(0,x) = 0,\\
u_1(t,1) = 0,\\
	u_1(t,0)-u^w_1(t,0) = \sqrt{2}k_0\varepsilon\pd{u_1}{x_2}(t,0) 
	+ 2k_2\varepsilon^2\pd{^2u_1}{x_2^2}(t,0),
\end{gather*}
where we let $\gamma_1=1$, $x_2\in[0,1]$ and $t\in[0,T]$.

We use an implicit finite difference method to solve this 
parabolic equation. Assume $x_i=ih$ where $0\leq i\leq N$ and
$h=1/N$, $t_n=n\Delta t$, $u_i^n:=u_1(t_n,x_i)$, then the numerical
scheme is 
\[
	\frac{u_i^{n+1}-u_i^n}{\Delta t} = \varepsilon\gamma_1
	\frac{u_{i+1}^{n+1}-2u_i^{n+1}+u_{i-1}^{n+1}}{h^2}.
\]
As for the boundary, the derivatives are approximated by
the one-sided difference.
It's well-known that the implicit scheme enlarges the 
feasible time step.

For the BGK model in the Couette flow, the linear moment equations are
\begin{gather*}
	\pd{W_c}{t} + A_c \pd{W_c}{x} = -\frac{1}{\varepsilon}Q_cW_c,\\
	W_c(0,x) = 0, \\
	W_c(t,1) = 0, \\
	B_cW_c(t,0) = b_c(t),
\end{gather*}
where $W_c,A_c,Q_c,B_c,b_c$ are defined in Section \ref{sec:3}.

This is a linear hyperbolic initial-boundary value problem with
constant coefficients. We use the upwind scheme to approximate
the convection term and implicitly deal with the source term:
\[
	\frac{W_i^{n+1}-W_i^n}{\Delta t} + \frac{\mathcal{F}_{i+1/2}^n
	-\mathcal{F}_{i-1/2}^n}{h} = -\frac{1}{\varepsilon}Q_cW_i^{n+1},	
\]
where $W_i^n:=W_c(t_n,x_i)$ and $\mathcal{F}_{i\pm 1/2}^n$ is the 
upwinding numerical flux \cite{Leveque1992numerical}. Assume the
decomposition $A_c=R\Lambda R^{-1}$ with $\Lambda=\Lambda^++\Lambda^-$
where $\Lambda^+$ is diagonal with positive entries and $\Lambda^-$
is diagonal with negative entries. Then the upwinding numerical
flux is defined as
\[\mathcal{F}_{i+1/2}^n = A_c^+ W_i^n+A_c^- W_{i+1}^n,\]
where $A_c^+=R\Lambda^+R^{-1}$ and $A_c^-=R\Lambda^-R^{-1}$.
We also deal with the boundary condition according to the 
characteristic information, which is illustrated in Section
\ref{sec:3}. The time step is determined by the CFL condition.

In the numerical test, we choose $N=10000$ and
\[ u^w_1(t,0) = 1-\cos(2\pi t).\]

\begin{figure}[!htb]
\pgfplotsset{width=0.45\textwidth}
\centering
\begin{tikzpicture} 
\begin{axis}[
	xlabel=$x_2/\sqrt{\varepsilon}$, 
    ylabel=$u_1$,
    xmax = 1.0,	
    tick align=outside, 
    legend style={at={(0.68,0.95)},anchor=north} 
    ]
\addplot[smooth,blue,thick] table {\dpath/eps1.0e-01_T0.10_MoM.dat};
\addlegendentry{MoM}
\addplot[smooth,dashed,black,thick] table 
	{\dpath/eps1.0e-01_T0.10_u1.dat};
\addlegendentry{no-slip}
	\addplot[smooth,purple,thick] table {\dpath/eps1.0e-01_T0.10_u2.dat};
\addlegendentry{$1$st-slip}
	\addplot[smooth,red,thick] table {\dpath/eps1.0e-01_T0.10_u3.dat};
\addlegendentry{$2$ed-slip}
\end{axis}
\end{tikzpicture}
\hskip 3pt
\begin{tikzpicture} 
\begin{axis}[
	xlabel=$x_2/\sqrt{\varepsilon}$, 
    ylabel=$u_1$,
    xmax = 1.0,	
    tick align=outside, 
    legend style={at={(0.68,0.95)},anchor=north} 
    ]
\addplot[smooth,blue,thick] table {\dpath/eps1.0e-01_T0.25_MoM.dat};
\addlegendentry{MoM}
\addplot[smooth,dashed,black,thick] table 
	{\dpath/eps1.0e-01_T0.25_u1.dat};
\addlegendentry{no-slip}
	\addplot[smooth,purple,thick] table {\dpath/eps1.0e-01_T0.25_u2.dat};
\addlegendentry{$1$st-slip}
	\addplot[smooth,red,thick] table {\dpath/eps1.0e-01_T0.25_u3.dat};
\addlegendentry{$2$ed-slip}
\end{axis}
\end{tikzpicture}

\caption{The velocity profile when $\varepsilon=0.1$. 
	Left: $t=0.1$. Right: $t=0.25$.}
	\label{fig:001}
\end{figure}
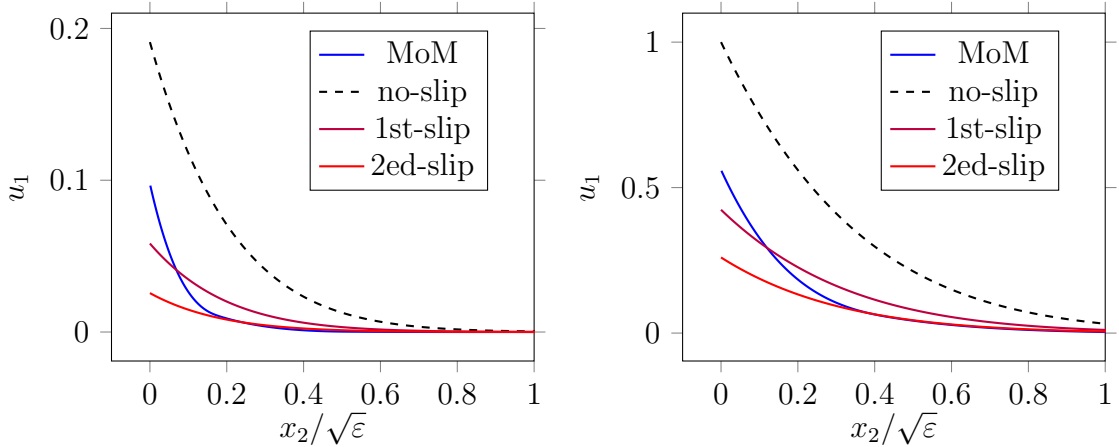

\begin{figure}[!htb]
\pgfplotsset{width=0.45\textwidth}
\centering
\begin{tikzpicture} 
\begin{axis}[
	xlabel=$x_2/\sqrt{\varepsilon}$, 
    ylabel=$u_1$,
    xmax = 1.0,	
    tick align=outside, 
    legend style={at={(0.68,0.95)},anchor=north} 
    ]
\addplot[smooth,blue,thick] table {\dpath/eps5.0e-02_T0.10_MoM.dat};
\addlegendentry{MoM}
\addplot[smooth,dashed,black,thick] table 
	{\dpath/eps5.0e-02_T0.10_u1.dat};
\addlegendentry{no-slip}
	\addplot[smooth,purple,thick] table {\dpath/eps5.0e-02_T0.10_u2.dat};
\addlegendentry{$1$st-slip}
	\addplot[smooth,red,thick] table {\dpath/eps5.0e-02_T0.10_u3.dat};
\addlegendentry{$2$ed-slip}
\end{axis}
\end{tikzpicture}
\hskip 3pt
\begin{tikzpicture} 
\begin{axis}[
	xlabel=$x_2/\sqrt{\varepsilon}$, 
    ylabel=$u_1$,
    xmax = 1.,	
    tick align=outside, 
    legend style={at={(0.68,0.95)},anchor=north} 
    ]
\addplot[smooth,blue,thick] table {\dpath/eps5.0e-02_T0.25_MoM.dat};
\addlegendentry{MoM}
\addplot[smooth,dashed,black,thick] table 
	{\dpath/eps5.0e-02_T0.25_u1.dat};
\addlegendentry{no-slip}
	\addplot[smooth,purple,thick] table {\dpath/eps5.0e-02_T0.25_u2.dat};
\addlegendentry{$1$st-slip}
	\addplot[smooth,red,thick] table {\dpath/eps5.0e-02_T0.25_u3.dat};
\addlegendentry{$2$ed-slip}
\end{axis}
\end{tikzpicture}

\caption{The velocity profile when $\varepsilon=0.05$. 
	Left: $t=0.1$. Right: $t=0.25$.}
	\label{fig:002}
\end{figure}
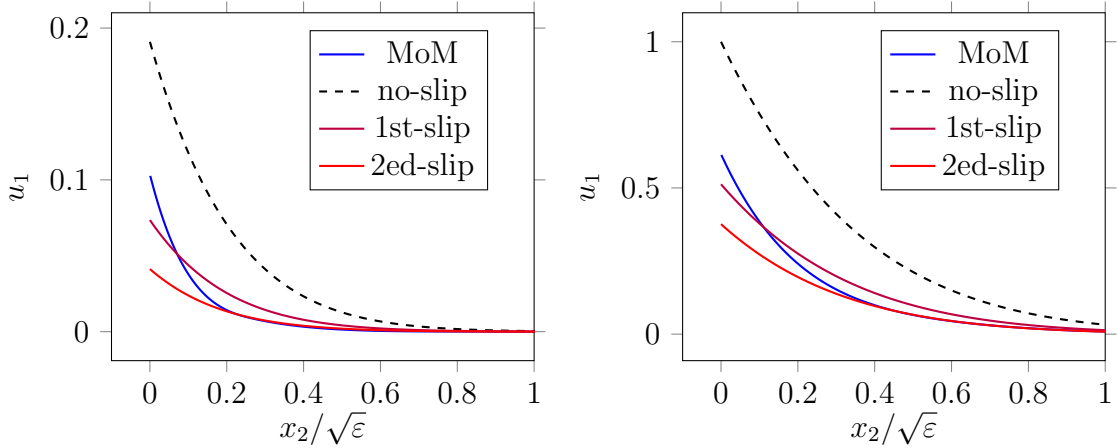

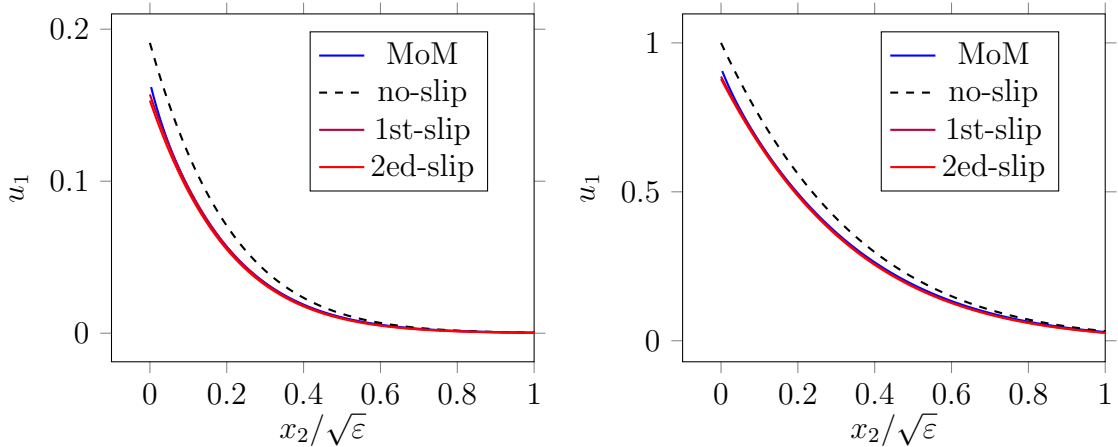
\begin{figure}[!htb]
\pgfplotsset{width=0.45\textwidth}
\centering
\begin{tikzpicture} 
\begin{axis}[
	xlabel=$x_2/\sqrt{\varepsilon}$, 
    ylabel=$u_1$,
    xmax = 1.0,	
    tick align=outside, 
    legend style={at={(0.68,0.95)},anchor=north} 
    ]
\addplot[smooth,blue,thick] table {\dpath/eps1.0e-03_T0.10_MoM.dat};
\addlegendentry{MoM}
\addplot[smooth,dashed,black,thick] table 
	{\dpath/eps1.0e-03_T0.10_u1.dat};
\addlegendentry{no-slip}
	\addplot[smooth,purple,thick] table {\dpath/eps1.0e-03_T0.10_u2.dat};
\addlegendentry{$1$st-slip}
	\addplot[smooth,red,thick] table {\dpath/eps1.0e-03_T0.10_u3.dat};
\addlegendentry{$2$ed-slip}
\end{axis}
\end{tikzpicture}
\hskip 3pt
\begin{tikzpicture} 
\begin{axis}[
	xlabel=$x_2/\sqrt{\varepsilon}$, 
    ylabel=$u_1$,
    xmax = 1.0,	
    tick align=outside, 
    legend style={at={(0.68,0.95)},anchor=north} 
    ]
\addplot[smooth,blue,thick] table {\dpath/eps1.0e-03_T0.25_MoM.dat};
\addlegendentry{MoM}
\addplot[smooth,dashed,black,thick] table 
	{\dpath/eps1.0e-03_T0.25_u1.dat};
\addlegendentry{no-slip}
	\addplot[smooth,purple,thick] table {\dpath/eps1.0e-03_T0.25_u2.dat};
\addlegendentry{$1$st-slip}
	\addplot[smooth,red,thick] table {\dpath/eps1.0e-03_T0.25_u3.dat};
\addlegendentry{$2$ed-slip}
\end{axis}
\end{tikzpicture}

\caption{The velocity profile when $\varepsilon=0.001$. 
	Left: $t=0.1$. Right: $t=0.25$.}
	\label{fig:003}
\end{figure}

\begin{figure}[!htb]
\pgfplotsset{width=0.45\textwidth}
\centering
\begin{tikzpicture} 
\begin{axis}[
	xlabel=$-\log_2\varepsilon$, 
    ylabel=$\log_2 u_e$,
    tick align=outside, 
    legend style={at={(0.68,0.95)},anchor=north} 
    ]
\addplot[smooth,blue,thick] table {\dpath/err23_T0.10.dat};
	\addlegendentry{+$2$ed-slip}
\addplot[smooth,dashed,red,thick] table {\dpath/err12_T0.10.dat};
	\addlegendentry{+$1$st-slip}
\end{axis}
\end{tikzpicture}
\hskip 3pt
\begin{tikzpicture} 
\begin{axis}[
	xlabel=$-\log_2\varepsilon$, 
    ylabel=$\log_2 u_e$,
    tick align=outside, 
    legend style={at={(0.68,0.95)},anchor=north} 
    ]
\addplot[smooth,blue,thick] table {\dpath/err23_T0.25.dat};
	\addlegendentry{+$2$ed-slip}
\addplot[smooth,dashed,red,thick] table {\dpath/err12_T0.25.dat};
	\addlegendentry{+$1$st-slip}
\end{axis}
\end{tikzpicture}

\caption{The log-log diagram of the velocity errors. Left: $t=0.1$.
	Right: $t=0.25$.}
	\label{fig:004}
\end{figure}
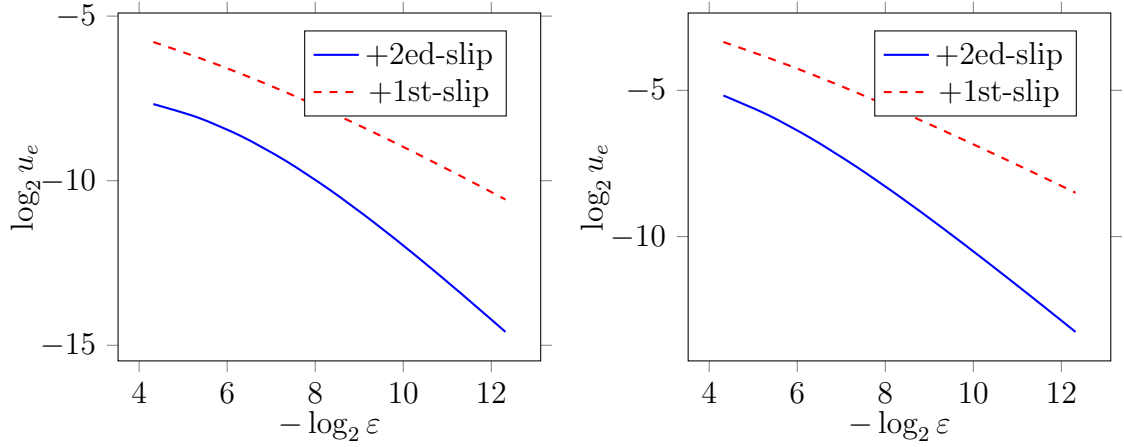

Fig.\ref{fig:001}\ -\ Fig.\ref{fig:003} show the velocity profile
when $\varepsilon=0.1,\ 0.05$ and $0.001$. In these figures, 
the $x$-axis is $x_2/\sqrt{\varepsilon}$, which means
zooming in to observe the solution behavior near the boundary. 
The label ``no-slip'' represents the solution of the NS
equations with $k_0=0$ and $k_2=0,$ while the label ``1st-slip'' 
as well as ``2ed-slip'' individually corresponds to
$k_0=1.01619,\ k_2=0$ and $k_0=1.01619,\ k_2=-0.76632.$ The label
``MoM'' represents the solution of the moment equations when $M=8$.

From Fig.\ref{fig:001}, we can see that the ``MoM'' solution deviates
from the NS solutions when $x_2$ is close to zero, and agrees with 
the ``2ed-slip'' solution when $x_2$ has a distance from zero.
The phenomena coincide with the theory: firstly, the NS equations can
not capture the Knudsen layer but the moment equations can, so when
$x_2=O(\varepsilon)$, the ``MoM''
solution would deviate from the NS solutions. Secondly, when
$x_2=O(\sqrt{\varepsilon})$, the moment equations can capture 
the viscous layer, which is also well approximated by the NS equations
with second-order BCs. So the ``MoM'' solution 
agrees with the ``2ed-slip'' solution when $x_2=O(\sqrt{\varepsilon})$.

From Fig.\ref{fig:001}\ -\ Fig.\ref{fig:003}, we 
can see that when $\varepsilon$ goes smaller, the difference 
between the ``1st-slip'' and ``2ed-slip'' solutions tends to
vanish. To show the error between NS solutions with different 
BCs, we exhibit Fig.\ref{fig:004}.

In Fig.\ref{fig:004}, the label ``+1st-slip'' means the $L^2$ error
of ``no-slip'' solution and ``1st-slip'' solution while the label
``+2ed-slip'' representing the error between ``1st-slip'' and
``2ed-slip'' solutions. We can see that the ``+1st-slip'' error
roughly has a convergence rate of the order $1/2$ while the
``+2ed-slip'' error is about first order. The fact implies that
if the $L^2$ error between the viscous layer solution of the moment
equations and the ``2ed-slip'' solution is $o(\varepsilon)$, 
then the error between the moment solution and the ``1st-slip'' solution
should be at most $O(\varepsilon)$.
Thus, for instantaneous flows, the second-order slip 
BCs for the NS equations are necessary to obtain a 
first-order approximation solution to the moment equations.

In conclusion, the well-designed numerical example of the unsteady 
Couette flow verifies the theoretical results in Section \ref{sec:3}.

%% file: conclusion.tex
\section*{Acknowledgment}
\textbf{Funding:}
This work is supported by the National Key R\&D
Program of China, Project Number 2020YFA0712000 and the China Postdoctoral Science Foundation, Project Number 2021M700002. \\
\textbf{Data Availibility:}
Data sharing not applicable to this article as no datasets were generated or analysed during the current study.\\
\textbf{Conflict of interest:} The authors have no financial or proprietary interests in any material discussed in this article.


%% file: App.tex
\begin{appendices}
\section{Derivation of the Grad BCs}
	\label{app:G}
	Multiplying the Maxwell BC \eqref{eq:Maxwell} by
	$\xi_2\phi_{\+\alpha}$ with $\+\alpha\in\mathbb{I}_e,
	\ |\+\alpha|\leq M-1,$ and integrating over $\+\xi\in \mathbb{R}^3$, 
	we have 
 \[
	 \int_{\bbR^2}\!\!\int_0^{+\infty}\!\!\xi_2
	 \phi_{\+\alpha}f(t,\+x,\+\xi)\,\mathrm{d}\+\xi = 
	 \chi\int_{\bbR^2}\!\!\int_0^{+\infty}\!\!\xi_2
	 \phi_{\+\alpha}f^w(t,\+x,\+\xi)\,\mathrm{d}\+\xi +
	 (1-\chi)\int_{\bbR^2}\!\!\int_0^{+\infty}\!\!\xi_2
	 \phi_{\+\alpha}f(t,\+x,\+\xi^*)\,\mathrm{d}\+\xi.
 \]
	Then we will substitute $f=\mathcal{M}\sum_{|\+\alpha|\leq M}
	w_{\+\alpha}\phi_{\+\alpha}$ into the above relation.

	In the process, the half-space integral is transformed into
	the whole space integral according to the even-odd parity.
	For example, when $\alpha_2$ is even and $\beta_2$ is odd,
	the polynomial $\xi_2\phi_{\+\alpha}\phi_{\+\beta}$ is even
	and we have
\begin{eqnarray*}
	\int_{\bbR^2}\!\!\int_0^{+\infty}\!\!\mathcal{M}\xi_2
	 \phi_{\+\alpha}\phi_{\+\beta}\,\mathrm{d}\+\xi &=&
	\frac{1}{2}\ang{\mathcal{M}\xi_2\phi_{\+\alpha}\phi_{\+\beta}}.
\end{eqnarray*}
	When $\alpha_2$ and $\beta_2$ are both even, the polynomial
	$\xi_2\phi_{\+\alpha}\phi_{\+\beta}$ is odd and we have
\begin{eqnarray*}
	\int_{\bbR^2}\!\!\int_0^{+\infty}\!\!\mathcal{M}\xi_2
	 \phi_{\+\alpha}\phi_{\+\beta}\,\mathrm{d}\+\xi &=&
	\frac{1}{2}\ang{\mathcal{M}|\xi_2|\phi_{\+\alpha}\phi_{\+\beta}}.
\end{eqnarray*}
	Note that $\+\xi^*=(\xi_1,-\xi_2,\xi_3)$. We also calculate
	\[ b_{\+\alpha}=\ang{f^w\phi_{\+\alpha}},\]
	and substitute $f^w=\mathcal{M}\sum_{|\+\alpha|\leq M}b_{\+\alpha}
	\phi_{\+\alpha}$ into the above relation. According to
	the calculation in \cite{CaiLiQiao}, we have
	\[b_{\+0}=\rho^w,\ b_{\+e_i}=u_i^w,\ b_{2\+e_i}=\theta^w/\sqrt{2},
	\text{\ otherwise}\ b_{\+\alpha}=0.\]

	After the above manipulation, the Grad BCs read as
\begin{equation}\label{eq:rbc1}
	\left(1-\frac{\chi}{2}\right)\sum_{\+\beta\in\mathbb{I}_o}\ang
	{\xi_2\mathcal{M}\phi_{\+\alpha}\phi_{\+\beta}} w_{\+\beta}
	= -\frac{\chi}{2}\sum_{\+\beta\in\mathbb{I}_e}
	\ang{|\xi_2|\mathcal{M}
	\phi_{\+\alpha}\phi_{\+\beta}}(w_{\+\beta}-b_{\+\beta}),
\end{equation}
where $\+\alpha\in\mathbb{I}_e$ and $|\+\alpha|\leq M-1$.
To rewrite \eqref{eq:rbc1} into a matrix form, we further
define two mappings 
$$
\mathcal{N}_1:\mathbb{I}_e\rightarrow \{1,2,...,m\},\qquad
\mathcal{N}_2:\mathbb{I}_o\rightarrow\{1,2,...,n\}
$$ 
by the relation 
\begin{align*}
&\mathcal{N}_1(\+\alpha) < \mathcal{N}_1(\+\beta)
	\quad\Leftrightarrow\quad 
	\mathcal{N}(\+\alpha) < \mathcal{N}(\+\beta),
	\quad \text{for}~\+\alpha,\+\beta\in \mathbb{I}_e, \\[2mm]
&\mathcal{N}_2(\+\alpha) < \mathcal{N}_2(\+\beta)
	\quad\Leftrightarrow\quad 
	\mathcal{N}(\+\alpha) < \mathcal{N}(\+\beta),
	\quad \text{for}~\+\alpha,\+\beta\in \mathbb{I}_o. 
\end{align*} 
Having these, we define $M_o\in\bbR^{m\times n}$ and
$S\in\bbR^{m\times m}$ by
\begin{gather*}
M_o[\mathcal{N}_1(\+\alpha),\mathcal{N}_2(\+\beta)] = \ang{\xi_2\mathcal{M}\phi_{\+\alpha}
\phi_{\+\beta}},\quad \+\alpha\in\mathbb{I}_e,\ \+\beta\in
\mathbb{I}_o.\\[2mm]
S[\mathcal{N}_1(\+\alpha),\mathcal{N}_1(\+\beta)] = \frac{\sqrt{2\pi}}{2}
\ang{|\xi_2|\mathcal{M}\phi_{\+\alpha}
\phi_{\+\beta}},\quad \+\alpha,\ \+\beta\in\mathbb{I}_e.
\end{gather*}
Then Grad BCs can be written as 
\begin{equation*}
	E[\hat{\chi}S,M_o](W-b)=0,
\end{equation*}
where $E\in\bbR^{n\times m}$ with
\[ E[\mathcal{N}_2(\+\alpha),\mathcal{N}_1(\+\beta)]
=\delta_{\+\alpha,\+\beta+\+e_2},\quad \+\alpha\in\mathbb{I}_o,
\ \+\beta\in\mathbb{I}_e.\]

\section{Details of the asymptotic analysis}
	\label{app:A}
\subsection{Outer solution}
	The choices of $G$ and $H$ in Section \ref{sec:401} are not
	unique. Apparently, the different choices will give equivalent
	equations in the linear case. Here we will give a specific choice
	of $G$ and $H$.

	Due to the collision invariants of the linearized Boltzmann
	operator \cite{CC1989}, the null space of $Q$ always has a
	constant dimension for any moment order $M\geq 3$, i.e.,
	\[p=\mathrm{dim}\ \mathrm{Null}\{Q\}=5,\]
	which corresponds to the conservation of mass, momentum and
	energy. Thus, here and hereafter, we choose $G\in\bbR^{N\times p}$
	as 
	\[G=[\varphi_0,\varphi_1,\varphi_2,\varphi_3,\varphi_4],\]
	where $\varphi_i\in\bbR^N$ is given with the non-zero entries
	\begin{eqnarray*}
		\varphi_0[\mathcal{N}(\+0)] &=& 1,\\
		\varphi_i[\mathcal{N}(\+e_i)] &=& 1,\quad i=1,2,3,\\
		\varphi_4[\mathcal{N}(2\+e_i)] &=& \sqrt{3}/3.
	\end{eqnarray*}
	Incidentally, we can define
	\[G_e=[\varphi_0,\varphi_1,\varphi_3,\varphi_4].\]

	The equilibrium variables $G^TW$ is 
	\[
		G^TW = \left[\rho,u_1,u_2,u_3,
		\displaystyle\frac{\sqrt{6}}{2}\theta\right].
	\]	
	Due to the structure of $G$, we can choose $H\in\bbR^{N\times(N-p)}$ 
	such that the columns of $H$ are all unit vectors with only 
	one component being one except two columns, i.e.,
	\[
		H[\mathcal{N}(\+\alpha),\mathcal{N}(\+\beta)-5]	= 
		\delta_{\+\alpha,\+\beta},\quad |\+\beta|>1,\ 
		\+\beta\neq 2\+e_1,2\+e_2,2\+e_3,
	\]
	and
	\begin{eqnarray*}
		H[\mathcal{N}(\+\alpha),\mathcal{N}(2\+e_2)-5]	= 
		\delta_{\+\alpha,2\+e_1}\frac{\sqrt{3}}{3} +
		\delta_{\+\alpha,2\+e_2}\frac{-3-\sqrt{3}}{6} +
		\delta_{\+\alpha,2\+e_3}\frac{3-\sqrt{3}}{6},\\
		H[\mathcal{N}(\+\alpha),\mathcal{N}(2\+e_3)-5]	= 
		\delta_{\+\alpha,2\+e_1}\frac{\sqrt{3}}{3} +
		\delta_{\+\alpha,2\+e_2}\frac{3-\sqrt{3}}{6} +
		\delta_{\+\alpha,2\+e_3}\frac{-3-\sqrt{3}}{6}.
	\end{eqnarray*}

	In the derivation, we multiply \eqref{eq:os} left by $G^T$, 
	where the right hand side item would be zero because
	\[ G^TQ = (QG)^T = 0.\]
	Multiply \eqref{eq:os} left by $H^T$, and we have
	\[ 0 = H^TQ\os{W}{0} = H^TQ(GG^T+HH^T)\os{W}{0} 
		= H^TQH(H^T\os{W}{0}).\]
	Since $H^TQH>0$, this formula implies that 
	\[H^T\os{W}{0}=0,\]
	which leads to $\os{\sigma}{0}{id}=0$ and $\os{q}{0}{d}=0$.
	Analogously, we have $H^T\os{W}{1}=0$. The general
	iteration is \eqref{eq:osnc}, which is similar as the 
	Maxwell iteration of moment equations \cite{Rein1990}. 
	Utilizing the symmetry properties of $Q$,
	one can get the constants in \eqref{eq:cc01}\eqref{eq:cc02} as
	\begin{eqnarray}\label{eq:cons0}
		\gamma_1 = r_{12}^T(H^TQH)^{-1}r_{12},\\ 
		\gamma_2 = \frac{1}{5}s_1^T(H^TQH)^{-1}s_1,\label{eq:cons1}
	\end{eqnarray}
where non-zero entries of $r_{id}\in\bbR^{N-p}$ and
$s_d\in\bbR^{N-p},\ d=1,2,3,$ are
\begin{gather*}
	r_{id}[\mathcal{N}(\+e_i+\+e_d)-5]=1,\
	i\neq 1\text{\ or\ }d\neq1,\\ 
	s_d[\mathcal{N}(3\+e_d)-5]=\sqrt{3/2},\ 
	s_d[\mathcal{N}(\+e_d+2\+e_i)-5]=\sqrt{1/2},\ i\neq d.
\end{gather*}
More details about $\gamma_1$ and $\gamma_2$ can be found in
\cite{Rein1990}, which compares the results of the moment equations
and Boltzmann equation.

\subsection{Viscous layer solution}
Multiplying \eqref{eq:vs} left by $H^T,$ we have
\begin{eqnarray*}
H^T\vs{W}{0}&=&0,\\
-(H^TQH)H^T\vs{W}{1} &=& H^TA_2(GG^T+HH^T)\pd{\vs{W}{0}}{y} \\
					&=& H^TA_2GG^T\pd{\vs{W}{0}}{y},\\
	-(H^TQH)H^T\vs{W}{j+2} &=& 
	H^T\left(\pd{\vs{W}{j}}{t} +\sum_{d\neq
		2}A_d\pd{\vs{W}{j}}{x_d}\right)
	+ H^TA_2\pd{\vs{W}{j+1}}{y} \\
	&=& H^T\sum_{d\neq 2}A_dGG^T\pd{\vs{W}{j}}{x_d} + 
	H^TA_2(GG^T+HH^T)\pd{\vs{W}{j+1}}{y},\quad j\geq 0.
\end{eqnarray*}
The first formula shows that $\vs{\sigma}{0}{id}=\vs{q}{0}{d}=0.$ 
From the second formula, the non-equilibrium variable $H^T\vs{W}{1}$
is represented by derivatives of $G^T\vs{W}{0}$. From the third
formula, the variable $H^T\vs{W}{j+2}$ can be represented by 
derivatives of $G^T\vs{W}{j+1},\ G^T\vs{W}{j}$ and $H^T\vs{W}{j+1}$.
By induction, we conclude that $H^T\vs{W}{j}$ can be represented
by derivatives of $G^T\vs{W}{s},\ s<j.$

For example, $\vs{\sigma}{2}{11}$ can be represented by the
linear combination of components of $H^T\vs{W}{2}$. Since all
matrices are known, with the aid of the computer algebra system,
we can calculate that
\[
	\vs{\sigma}{2}{11} = -\gamma_1\left(2\pd{\vs{u}{0}{1}}{x_1}
			-\frac{2}{3}
			\left(\pd{\vs{u}{0}{1}}{x_1}+\pd{\vs{u}{0}{3}}{x_3}+
			\pd{\vs{u}{1}{2}}{y}\right)\right)
		 - \frac{2}{3}\gamma_3
		\pd{^2\vs{\theta}{0}}{y^2},\ 
\]
where the constant
\begin{equation}
	\label{eq:Agam3}
	\gamma_3 = r_{12}^T(H^TQH)^{-1}(H^TA_2H)(H^TQH)^{-1}s_1.
\end{equation}
Due to the symmetry of $Q$, the same $\gamma_3$ appears in
other relations \cite{Sone2007}, e.g., $\vs{\sigma}{2}{id}$.

\subsection{Knudsen layer solution}
The Knudsen layer solution satisfies the linear half-space problem
\begin{eqnarray*}
	A_2\pd{\ks{W}{j}}{z} = -Q\ks{W}{j}-\pd{\ks{W}{j-2}}{t}
	-\sum_{d\neq 2}\pd{\ks{W}{j-2}}{x_d},\quad z\geq 0,
\end{eqnarray*}
where $\ks{W}{j}$ vanishes at $z=+\infty.$ Bobylev and Bernhoff
\cite{Bern2008} have studied the number of positive, negative and zero
eigenvalues of $A_2^{-1}Q$ when $A_2$ is invertible. They also
extend the result for singular matrices $A_2$ when $\mathrm{Null}(A_2)
\cap\mathrm{Null}(Q)=\{0\}$.

In the Grad moment equations, the matrix $A_2$ has a special block
structure which is called the Onsager matrix \cite{Sarna2018} by 
some authors, i.e.,
\[A_2=\begin{bmatrix} 0 & M_o \\ M_o^T & 0 \end{bmatrix},\]
where $M_o\in\bbR^{m\times n}$ is of full column rank.
It's easy to see that
	\[\begin{bmatrix} 0 & M_o \\ M_o^T & 0 \end{bmatrix}
		\begin{bmatrix} x \\ y \end{bmatrix} = \lambda
	\begin{bmatrix} x \\ y \end{bmatrix} \quad\Rightarrow\quad
	\begin{bmatrix} 0 & M_o \\ M_o^T & 0 \end{bmatrix}
	   	\begin{bmatrix} x \\ -y \end{bmatrix}
	   	= -\lambda \begin{bmatrix} x \\ -y \end{bmatrix}.\]
So $A_2$ has $n$ positive eigenvalues, $n$ negative eigenvalues and
$m-n$ zero eigenvalues. The matrix $Q$ can write as
\[Q=\begin{bmatrix} Q_e & \\ & Q_o \end{bmatrix},\]
where $Q_e\in\bbR^{m\times m}$ and $Q_o\in\bbR^{n\times n}$.
We can check that for the linear Grad moment equations, 
\[\mathrm{dim}\ \mathrm{Null}(Q_e) = 4.\]
Utilizing Bobylev and Bernhoff's result or our specific version for the
moment equations \cite{Yang2022a}, we can verify that the half-space
problem need $n-4$ boundary conditions at $z=0.$

\section{Numerical method of the half-space problem}
\label{app:C}
Now we show the numerical method to solve the half-space problem
\begin{eqnarray}\label{eq:HS1}
	\begin{aligned}
		A_2\pd{\ks{W}}{z} = -Q\ks{W},& \quad \ks{W}=\ks{W}(z),\ z\geq 0, \\
		B(\ks{W}(0)-h) = 0,&\quad \ks{W}(\infty) = 0,
	\end{aligned}
\end{eqnarray}
where the matrices $A_2, Q$ and $B$ are given by
\eqref{eq:defofAQ} and \eqref{eq:GBC}.

The numerical method is briefly stated as follows. Since $A_2$
is symmetric and $Q$ is symmetric positive semi-definite, we
can solve a generalized eigenvalue problem to get
\[ A_2 x_i = \lambda_i Q x_i,\ x_i\in\bbR^N,\ \lambda_i\in\bbR\cup
\{\infty\},\ i=1,2,...,N.\]
Because $\ks{W}$ vanishes at infinity, the characteristic variables
corresponding to non-positive eigenvalues should be zero, i.e.,
\[ x_i^TQ\ks{W} = 0,\ \lambda_i\leq 0 \ \text{or}\ \lambda_i=\infty.\]
The above formula gives a relation between components of $\ks{W}$, i.e.,
only part of components of $\ks{W}$ is independent.

Substitute the relation into BCs and write $h$ as
$h=h_+ +h_-$, where $h_-$ represents the given driven term, then
we may solve a linear algebraic system to determine $h_+$.
The solvability of this problem is ensured by Theorem \ref{thm:41}
when $h_-$ is appropriately chosen.
More details about the numerical method can be found in
\cite{Yang2022a}.

\end{appendices}